\documentclass[10pt]{article}
\usepackage{epsf}
\usepackage{amsmath}

\allowdisplaybreaks

\usepackage[showframe=false]{geometry}
\usepackage{changepage}

\usepackage{epsfig}
\usepackage{amssymb}

\usepackage{amsthm}
\usepackage{setspace}
\usepackage{cite}
\usepackage{mcite}

\usepackage{algorithmic}  
\usepackage{algorithm}

\usepackage{shadow}
\usepackage{fancybox}
\usepackage{fancyhdr}

\usepackage{color}
\usepackage[usenames,dvipsnames,svgnames,table]{xcolor}
\newcommand{\bl}[1]{\textcolor{blue}{#1}}

\definecolor{mypurple}{rgb}{.4,.0,.5}
\newcommand{\prp}[1]{\textcolor{mypurple}{#1}}

\usepackage[hyphens]{url}

\usepackage[colorlinks=true,
            linkcolor=black,
            urlcolor=blue,
            citecolor=purple]{hyperref}

\usepackage{breakurl}

\def\y{{\bf y}}

\def\x{{\bf x}}

\def\x{{\mathbf x}}

\def\u{{\bf u}}

\def\x{{\bf x}}
\def\y{{\bf y}}

\def\h{{\bf h}}

\def\be{\begin{equation}}
\def\ee{\end{equation}}
\def\ba{\left[\begin{array}}
\def\ea{\end{array}\right]}

\def\u{{\bf u}}

\def\x{{\bf x}}
\def\y{{\bf y}}

\def\1{{\bf 1}}

\def\0{{\bf 0}}

\def\calX{{\cal X}}
\def\calY{{\cal Y}}

\def\mR{{\mathbb R}}

\def\mE{{\mathbb E}}

\def\lp{\left (}
\def\rp{\right )}

\newtheorem{theorem}{Theorem}
\newtheorem{corollary}{Corollary}

\begin{document}

\begin{singlespace}

\title {Fully bilinear generic and lifted random processes comparisons 
}
\author{
\textsc{Mihailo Stojnic
\footnote{e-mail: {\tt flatoyer@gmail.com}} }}
\date{}
\maketitle

\centerline{{\bf Abstract}} \vspace*{0.1in}

In our companion paper \cite{Stojnicgscomp16} we introduce a collection of fairly powerful statistical comparison results.
They relate to a general comparison concept and its an upgrade that we call lifting procedure. Here we provide a different generic principle (which we call fully bilinear) that in certain cases turns out to be stronger than the corresponding one from \cite{Stojnicgscomp16}. Moreover, we also show how the principle that we introduce here can also be pushed through the lifting machinery of \cite{Stojnicgscomp16}. Finally, as was the case in \cite{Stojnicgscomp16}, here we also show how the well known Slepian's max and Gordon's minmax comparison principles can be obtained as special cases of the mechanisms that we present here. We also create their lifted upgrades which happen to be stronger than the corresponding ones in \cite{Stojnicgscomp16}. A fairly large collection of results obtained through numerical experiments is also provided. It is observed that these results are in an excellent agreement with what the theory predicts.

\vspace*{0.25in} \noindent {\bf Index Terms: Random processes; comparison principles, lifting}.

\end{singlespace}

\section{Introduction}
\label{sec:back}

The main topic of this paper are random processes comparisons. This topic has been studied for quite some time and many excellent results were obtained in various directions over the last half a century. In our view the major highlights that have found a large spectrum of applications are
the Slepian's max \cite{Slep62} and the Gordon's minmax \cite{Gordon85} principle (see also \cite{Sudakov71,Fernique74,Fernique75,Kahane86}). The list of applications in various fields is of course pretty much endless. As comparison principles are also the main topic of our companion paper \cite{Stojnicgscomp16} we will refrain from further detailing about their importance and the history of their development (more on this can be found in e.g. \cite{Adler90,Lifshits85,LedTal91,Tal05}). Instead, we here single out that, through studying the performance characterizations of many hard random optimization problems, we in recent years also fairly often utilized as the main probabilistic foundation the comparison principles (see, e.g. \cite{StojnicISIT2010binary,StojnicCSetam09,StojnicUpper10,StojnicCSetamBlock09,StojnicICASSP10knownsupp} and references therein). In fact, not only were our techniques strong enough to handle many of these problems they also turned out to be capable of doing it on an ultimate precision level. On the other hand, some of the results that we initially created in e.g. \cite{StojnicISIT2010binary,StojnicCSetam09,StojnicUpper10,StojnicCSetamBlock09,StojnicICASSP10knownsupp}, we later on managed to substantially upgrade (more on this can be found in, e.g. \cite{StojnicLiftStrSec13,StojnicMoreSophHopBnds10,StojnicRicBnds13} and references therein). The foundational blocks of these upgrades were actually rooted in core upgrades in the underlying random processes' comparisons.

As it will be rather clear on quite a few occasions throughout the paper, we view, the Slepian's max \cite{Slep62} and the Gordon's minmax \cite{Gordon85} comparison principles as two of the most influential results not only in the comparison theory but pretty much in a large section of the general probability theory. Both of them are derived basically starting almost from the axioms and with very minimal prior knowledge (a fairly short line of work, e.g. \cite{Schlafli858,Placket54,Chover61}, precedes Slepain's on the one hand and almost nothing besides Slepian's work precedes the direction of the Gotrdon's work on the other hand). In this paper we will deal with generic comparison principles that will not directly relate to the extrema of the random processes as, to a large degree, do Slepian's and Gordon's work. However, we will also show how easily a set of particularly useful forms of both of these classical achievements can be deduced from what we will present here.

In our companion paper \cite{Stojnicgscomp16} we also introduce a generic comparison principle that can be simplified in certain scenarios to include the above mentioned classical max and minmax forms. The mechanism that we introduce here is conceptually different and in certain cases of particular interest (such as dealing with the extrema of the random processes) it will produce a stronger set of results than those presented in \cite{Stojnicgscomp16}. Nonetheless, quite a few observations made in \cite{Stojnicgscomp16} will turn out to be of use here as well and we will try to follow the style of the presentation given in \cite{Stojnicgscomp16} so that all the similarities and differences are easier to see.

Along the same lines and following into the footsteps of \cite{Stojnicgscomp16}, we will split the presentation into two main parts: 1) the first part where we will discuss a generic comparison principle (to which we will refer as fully bilinear) and its connections with the well-known Slepian's max and Gordon's minmax principles; and 2) the second part where we will discuss a way to upgrade these generic methods through a lifting procedure similar to the one that we consider in \cite{Stojnicgscomp16}.

\section{A bilinear comparison form}
\label{sec:gencon}

We start with two given sets, say set $\calX=\{\x^{(1)},\x^{(2)},\dots,\x^{(l)}\}$, where $\x^{(i)}\in \mR^n,1\leq i\leq l$, and set $\calY=\{\y^{(1)},\y^{(2)},\dots,\y^{(l)}\}$, where $\y^{(i)}\in \mR^m,1\leq i\leq l$, and consider the following function
\begin{eqnarray}\label{eq:genanal1}
 f(G,u^{(4)},\calX,\calY,\beta,s)= \frac{1}{\beta|s|\sqrt{n}} \log\lp \sum_{i_1=1}^{l}\lp\sum_{i_2=1}^{l}e^{\beta \lp (\y^{(i_2)})^T
 G\x^{(i_1)}+\|\x^{(i_1)}\|_2\|\y^{(i_2)}\|_2 u^{(4)}\rp} \rp^{s}\rp,
\end{eqnarray}
where $s$ and $\beta>0$ are real parameters. Similarly to what we did in our companion paper \cite{Stojnicgscomp16}, we will study this function in a random medium. Namely, we will consider $(m\times n)$ dimensional matrices  $G\in \mR^{m\times n}$ with i.i.d. standard normal components. Moreover, we will assume that $u^{(4)}$ is also a standard normal random variable but independent of $G$. In such a random medium (and especially if the dimensions of $G$ are large) the expected value of the above function is usually its most relevant value. Let this expected value be $\xi(\calX,\calY,\beta,s)$. Then we set
\begin{eqnarray}\label{eq:genanal2}
\xi(\calX,\calY,\beta,s) & \triangleq  & \mE_{G,u^{(4)}} f(G,u^{(4)},\calX,\calY,\beta,s) \nonumber \\
& = & \mE_{G,u^{(4)}}\frac{1}{\beta|s|\sqrt{n}} \log\lp \sum_{i_1=1}^{l}\lp\sum_{i_2=1}^{l}e^{\beta \lp (\y^{(i_2)})^T
 G\x^{(i_1)}+\|\x^{(i_1)}\|_2\|\y^{(i_2)}\|_2 u^{(4)}\rp} \rp^{s}\rp.
\end{eqnarray}
Following into the footsteps of \cite{Stojnicgscomp16}, we will consider the following interpolating function $\psi(\cdot)$ as an object convenient for studying properties of $\xi(\calX,\calY,\beta,s)$
\begin{multline}\label{eq:genanal3}
\psi(\calX,\calY,\beta,s,t)  =  \mE_{G,u^{(4)},\u^{(2)},\h} \frac{1}{\beta|s|\sqrt{n}} \\
 \times \log\lp \sum_{i_1=1}^{l}\lp\sum_{i_2=1}^{l}e^{\beta \lp \sqrt{t}(\y^{(i_2)})^T
 G\x^{(i_1)}+\sqrt{1-t}\|\x^{(i_2)}\|_2 (\y^{(i_2)})^T\u^{(2)}+\sqrt{t}\|\x^{(i_1)}\|_2\|\y^{(i_2)}\|_2,u^{(4)} +\sqrt{1-t}\|\y^{(i_2)}\|_2\h^T\x^{(i)}\rp} \rp^{s}\rp.
\end{multline}
In (\ref{eq:genanal3}), $\u^{(2)}$ and $\h$ are $m$ and $n$ dimensional vectors of i.i.d standard normals, respectively; they are assumed to be independent of each other and of $G$ and $u^{(4)}$ ($\mE$ denotes the expectation with respect to any randomness under the expectation; sometimes $\mE$ will have a subscript to emphasize the underlying randomness). Clearly, $\xi(\calX,\calY,\beta,s)=\psi(\calX,\calY,\beta,s,1)$ and given that $\psi(\calX,\calY,\beta,s,0)$ is typically easier to study than $\psi(\calX,\calY,\beta,s,1)$ we will try to connect $\psi(\calX,\calY,\beta,s,1)$ to $\psi(\calX,\calY,\beta,s,0)$ as a way of connecting $\xi(\calX,\calY,\beta,s)$ to $\psi(\calX,\calY,\beta,s,0)$. We will find it convenient below to set
\begin{eqnarray}\label{eq:genanal4}
\u^{(i_1,1)} & =  & \frac{G\x^{(i_1)}}{\|\x^{(i_1)}\|_2} \nonumber \\
\u^{(i_1,3)} & =  & \frac{\h^T\x^{(i_1)}}{\|\x^{(i_1)}\|_2}.
\end{eqnarray}
Denoting by $G_{j,1:n}$ the $j$-th row of $G$ and by $\u_j^{(i_1,1)}$ the $j$-th component of $\u^{(i_1,1)}$ from (\ref{eq:genanal4}) we have
\begin{eqnarray}\label{eq:genanal5}
\u_j^{(i_1,1)} & =  & \frac{G_{j,1:n}\x^{(i_1)}}{\|\x^{(i_1)}\|_2},1\leq j\leq m.
\end{eqnarray}
Also, one trivially has for any fixed $i_1$ that the elements of $\u^{(i_1,1)}$, $\u^{(2)}$, and $\u^{(i_1,3)}$ are i.i.d. standard normals. (\ref{eq:genanal3}) can then be rewritten as
\begin{multline}\label{eq:genanal6}
\psi(\calX,\calY,\beta,s,t)  =  \mE_{G,u^{(4)},\u^{(2)},\h} \frac{1}{\beta|s|\sqrt{n}} \\
\times \log\lp \sum_{i_1=1}^{l}\lp\sum_{i_2=1}^{l}e^{\beta_{i_1} \lp \sqrt{t}(\y^{(i_2)})^T
 \u^{(i_1,1)}+\sqrt{1-t} (\y^{(i_2)})^T\u^{(2)} +\sqrt{t}\|\y^{(i_2)}\|_2,u^{(4)}+\sqrt{t}\|\y^{(i_2)}\|_2\u^{(i_1,3)}\rp} \rp^{s}\rp,
\end{multline}
where $\beta_{i_1}=\beta\|\x^{(i_1)}\|_2$. To facilitate the exposition we also set
\begin{eqnarray}\label{eq:genanal7}
B^{(i_1,i_2)} & \triangleq &  \sqrt{t}(\y^{(i_2)})^T\u^{(i_1,1)}+\sqrt{1-t} (\y^{(i_2)})^T\u^{(2)} \nonumber \\
A^{(i_1,i_2)} & \triangleq &  e^{\beta_{i_1}(B^{(i_1,i_2)}+\sqrt{t}\|\y^{(i_2)}\|_2 u^{(4)}+\sqrt{1-t}\|\y^{(i_2)}\|_2\u^{(i_1,3)})}\nonumber \\
C^{(i_1)} & \triangleq &  \sum_{i_2=1}^{l}A^{(i_1,i_2)}\nonumber \\
Z & \triangleq & \sum_{i_1=1}^{l}\lp\sum_{i_2=1}^{l}e^{\beta_{i_1} \lp \sqrt{t}(\y^{(i_2)})^T
 \u^{(i_1,1)}+\sqrt{1-t} (\y^{(i_2)})^T\u^{(2)} +\sqrt{t}\|\y^{(i_2)}\|_2 u^{(4)}+\sqrt{t}\|\y^{(i_2)}\|_2\u^{(i_1,3)}\rp} \rp^{s}\nonumber \\
 & = & \sum_{i_1=1}^{l} \lp \sum_{i_2=1}^{l} A^{(i_1,i_2)}\rp^s =\sum_{i_1=1}^{l}  (C^{(i_1)})^s.
\end{eqnarray}
It is now relatively easy to see that (\ref{eq:genanal6}) and (\ref{eq:genanal7}) give
\begin{eqnarray}\label{eq:genanal8}
\psi(\calX,\calY,\beta,s,t) & = &  \mE_{\u^{(i_1,1)},\u^{(2)},\u^{(i_1,3)},u^{(4)}} \frac{1}{\beta|s|\sqrt{n}} \log(Z).
\end{eqnarray}
Our main topic of studying below will be the properties of $\psi(\calX,\calY,\beta,s,t)$. In particular we will study its monotonicity and show that $\psi(\calX,\calY,\beta,s,t)$ is a non-increasing (basically decreasing) function of $t$. We start with the analysis of its derivative
\begin{eqnarray}\label{eq:genanal9}
\frac{d\psi(\calX,\calY,\beta,s,t)}{dt} & = &  \mE_{\u^{(i_1,1)},\u^{(2)},\u^{(i_1,3)},u^{(4)}} \frac{1}{\beta|s|\sqrt{n}} \log Z\nonumber \\
& = &  \mE_{\u^{(i_1,1)},\u^{(2)},\u^{(i_1,3)},u^{(4)}} \frac{1}{Z\beta|s|\sqrt{n}} \frac{d\lp \sum_{i_1=1}^{l} \lp \sum_{i_2=1}^{l} A^{(i_1,i_2)}\rp^s \rp }{dt}\nonumber \\
& = &  \mE_{\u^{(i_1,1)},\u^{(2)},\u^{(i_1,3)},u^{(4)}} \frac{s}{Z\beta|s|\sqrt{n}}  \sum_{i=1}^{l} (C^{(i_1)})^{s-1} \nonumber \\
& & \times \sum_{i_2=1}^{l}\beta_{i_1}A^{(i_1,i_2)}\lp \frac{dB^{(i_1,i_2)}}{dt}+\frac{\|\y^{(i_2)}\|_2 u^{(4)}}{2\sqrt{t}}-\frac{\|\y^{(i_2)}\|_2 \u^{(i_1,3)}}{2\sqrt{1-t}}\rp.
\end{eqnarray}
Utilizing (\ref{eq:genanal7}) we find
\begin{equation}\label{eq:genanal10}
\frac{dB^{(i_1,i_2)}}{dt} =   \frac{d\lp\sqrt{t}(\y^{(i_2)})^T\u^{(i_1,1)}+\sqrt{1-t} (\y^{(i_2)})^T\u^{(2)}\rp}{dt}=
\sum_{j=1}^{m}\lp \frac{\y_j^{(i_2)}\u_j^{(i_1,1)}}{2\sqrt{t}}-\frac{\y_j^{(i_2)}\u_j^{(2)}}{2\sqrt{1-t}}\rp.
\end{equation}
Combining (\ref{eq:genanal9}) and (\ref{eq:genanal10}) we obtain
\begin{eqnarray}\label{eq:genanal11}
\frac{d\psi(\calX,\calY,\beta,s,t)}{dt} & = &  \mE_{\u^{(i_1,1)},\u^{(2)},\u^{(i_1,3)},u^{(4)}} \frac{1}{\beta|s|\sqrt{n}} \log Z\nonumber \\
& = &  \mE_{\u^{(i_1,1)},\u^{(2)},\u^{(i_1,3)},u^{(4)}} \frac{1}{Z\beta|s|\sqrt{n}} \frac{d\lp \sum_{i_1=1}^{l} \lp \sum_{i_2=1}^{l} A^{(i_1,i_2)}\rp^s \rp }{dt}\nonumber \\
& = &  \mE_{\u^{(i_1,1)},\u^{(2)},\u^{(i_1,3)},u^{(4)}} \frac{s}{Z\beta|s|\sqrt{n}}  \sum_{i_1=1}^{l} (C^{(i_1)})^{s-1} \nonumber \\
& & \times \sum_{i_2=1}^{l}\beta_{i_1}A^{(i_1,i_2)}\lp \sum_{j=1}^{m}\lp \frac{\y_j^{(i_2)}\u_j^{(i_1,1)}}{2\sqrt{t}}-\frac{\y_j^{(i_2)}\u_j^{(2)}}{2\sqrt{1-t}}\rp+\frac{\|\y^{(i_2)}\|_2 u^{(4)}}{2\sqrt{t}}-\frac{\|\y^{(i_2)}\|_2 \u^{(i_1,3)}}{2\sqrt{1-t}}\rp.\nonumber \\
\end{eqnarray}
Each of the terms in the above sum we will handle separately. To do so and to facilitate the presentation as much as possible we will try to parallel what was done in \cite{Stojnicgscomp16}. The calculations though will be substantially different.

\subsection{Computing $\frac{d\psi(\calX,\calY,\beta,s,t)}{dt}$}
\label{sec:compderivative}

As mentioned above, we will separately handle all the terms appearing in (\ref{eq:genanal11}).

\subsubsection{Finding $\mE_{\u^{(i_1,1)},\u^{(2)},\u^{(i_1,3)},u^{(4)}}  \frac{(C^{(i_1)})^{s-1} A^{(i_1,i_2)}\u_j^{(i_1,1)}\y_j^{(i_2)}}{Z}$}
\label{sec:hand1}

We start with the following standard utilization of the Gaussian integration by parts.
\begin{multline}\label{eq:genanal12}
\mE_{\u^{(i_1,1)},\u^{(2)},\u^{(i_1,3)},u^{(4)}}  \frac{(C^{(i_1)})^{s-1} A^{(i_1,i_2)}\u_j^{(i_1,1)}\y_j^{(i_2)}}{Z} =
\mE (\sum_{p_1=1,p_1\neq i_1}^{l} \mE (\u_j^{(i_1,1)}\u_j^{(p_1,1)})\frac{d}{d\u_j^{(p_1,1)}}\lp \frac{(C^{(i_1)})^{s-1} A^{(i_1,i_2)}\y_j^{(i_2)}}{Z}\rp  \\
+\mE (\u_j^{(i_1,1)}\u_j^{(i_1,1)})\frac{d}{d\u_j^{(i_1,1)}}\lp \frac{(C^{(i_1)})^{s-1} A^{(i_1,i_2)}\y_j^{(i_2)}}{Z}\rp).
\end{multline}
Clearly, $\mE (\u_j^{(i_1,1)}\u_j^{(p_1,1)})=\frac{(\x^{(i_1)})^T\x^{(p_1)}}{\|\x^{(i_1)}\|_2\|\x^{(p_1)}\|_2}$ and we also have
\begin{multline}\label{eq:genanal13}
\mE_{\u^{(i_1,1)},\u^{(2)},\u^{(i_1,3)},u^{(4)}}  \frac{(C^{(i_1)})^{s-1} A^{(i_1,i_2)}\u_j^{(i_1,1)}\y_j^{(i_2)}}{Z} =
\mE (\sum_{p_1=1,p_1\neq i_1}^{l} \frac{(\x^{(i_1)})^T\x^{(p_1)}}{\|\x^{(i_1)}\|_2\|\x^{(p_1)}\|_2}\frac{d}{d\u_j^{(p_1,1)}}\lp \frac{(C^{(i_1)})^{s-1} A^{(i_1,i_2)}\y_j^{(i_2)}}{Z}\rp  \\
+\frac{(\x^{(i_1)})^T\x^{(i_1)}}{\|\x^{(i_1)}\|_2\|\x^{(i_1)}\|_2}\frac{d}{d\u_j^{(i_1,1)}}\lp \frac{(C^{(i_1)})^{s-1} A^{(i_1,i_2)}\y_j^{(i_2)}}{Z}\rp ).
\end{multline}
For $p_1\neq i_1$ we obtain the following
\begin{equation}\label{eq:genanal14}
\frac{d}{d\u_j^{(p_1,1)}}\lp \frac{(C^{(i_1)})^{s-1} A^{(i_1,i_2)}\y_j^{(i_2)}}{Z}\rp=(C^{(i_1)})^{s-1} A^{(i_1,i_2)}\y_j^{(i_2)}\frac{d}{d\u_j^{(p_1,1)}}\lp \frac{1}{Z}\rp=-\frac{(C^{(i_1)})^{s-1} A^{(i_1,i_2)}\y_j^{(i_2)}}{Z^2}\frac{dZ}{d\u_j^{(p_1,1)}}.
\end{equation}
Now we also have
\begin{multline}\label{eq:genanal14a}
\frac{dZ}{d\u_j^{(p_1,1)}}=\frac{d\sum_{i_1=1}^{l}  (C^{(i_1)})^s}{d\u_j^{(p_1,1)}}
=s\sum_{i_1=1}^{l}  (C^{(i_1)})^{s-1}\frac{d(C^{(i_1)})}{d\u_j^{(p_1,1)}}\\
=s\sum_{i_1=1}^{l}  (C^{(i_1)})^{s-1}\sum_{i_2=1}^{l}\frac{d(A^{(i_1,i_2)})}{d\u_j^{(p_1,1)}}
=s  (C^{(p_1)})^{s-1}\sum_{i_2=1}^{l}\frac{d(A^{(p_1,i_2)})}{d\u_j^{(p_1,1)}}.
\end{multline}
Moreover, from (\ref{eq:genanal7}) we have
\begin{equation}\label{eq:genanal15}
\frac{d B^{(p_1,i_2)}}{d\u_j^{(p_1,1)}} =  \y^{(i_2)}\sqrt{t},
\end{equation}
and then
\begin{equation}\label{eq:genanal16}
\frac{d(A^{(p_1,i_2)})}{d\u_j^{(p_1,1)}}=\beta_{p_1}A^{(p_1,i_2)}\frac{d(B^{(p_1,i_2)})}{d\u_j^{(p_1,1)}}
=\beta_{p_1}A^{(p_1,i_2)}\y_j^{(i_2)}\sqrt{t}.
\end{equation}
Combining (\ref{eq:genanal14}), (\ref{eq:genanal14a}), and (\ref{eq:genanal16}) we obtain
\begin{eqnarray}\label{eq:genanal17}
\frac{d}{d\u_j^{(p_1,1)}}\lp \frac{(C^{(i_1)})^{s-1} A^{(i_1,i_2)}\y_j^{(i_2)}}{Z}\rp & = & -\frac{(C^{(i_1)})^{s-1} A^{(i_1,i_2)}\y_j^{(i_2)}}{Z^2}
s  (C^{(p_1)})^{s-1}\sum_{p_2=1}^{l}\frac{d(A^{(p_1,i_2)})}{d\u_j^{(p_1,1)}}\nonumber \\
 & = & -\frac{(C^{(i_1)})^{s-1} A^{(i_1,i_2)}\y_j^{(i_2)}}{Z^2}
s  (C^{(p_1)})^{s-1}\sum_{p_2=1}^{l}\beta_{p_1}A^{(p_1,p_2)}\y_j^{(p_2)}\sqrt{t}.
\end{eqnarray}
For $p=i$ we have
\begin{eqnarray}\label{eq:genanal18}
\frac{d}{d\u_j^{(i_1,1)}}\lp \frac{(C^{(i_1)})^{s-1} A^{(i_1,i_2)}\y_j^{(i_2)}}{Z}\rp  & = &
\frac{\y_j^{(i_2)}}{Z}\frac{d}{d\u_j^{(i_1,1)}}\lp (C^{(i_1)})^{s-1} A^{(i_i,i_2)}\rp-\frac{ (C^{(i_1)})^{s-1} A^{(i_1,i_2)}\y_j^{(i_2)}}{Z^2}
\frac{dZ}{d\u_j^{(i_1,1)}}.\nonumber \\
\end{eqnarray}
From (\ref{eq:genanal14a}) and (\ref{eq:genanal17}) we have
\begin{multline}\label{eq:genanal18a}
\frac{dZ}{d\u_j^{(i_1,1)}}=\frac{d\sum_{i_1=1}^{l}  (C^{(i_1)})^s}{d\u_j^{(p_1,1)}}
=s  (C^{(i_1)})^{s-1}\sum_{p_2=1}^{l}\frac{d(A^{(i_1,p_2)})}{d\u_j^{(i_1,1)}}=s  (C^{(i_1)})^{s-1}\sum_{p_2=1}^{l}
\beta_{i_1}A^{(i_1,p_2)}\y_j^{(p_2)}\sqrt{t}.
\end{multline}
Also,
\begin{eqnarray}\label{eq:genanal18b}
\frac{d}{d\u_j^{(i_1,1)}}\lp (C^{(i_1)})^{s-1} A^{(i_1,i_2)}\rp
& = & (C^{(i_1)})^{s-1} \frac{dA^{(i_1,i_2)} }{d\u_j^{(i_1,1)}}+ A^{(i_i,i_2)}\frac{d(C^{(i_1)})^{s-1}}{d\u_j^{(i_1,1)}} \nonumber \\
& = & (C^{(i_1)})^{s-1}\beta_{i_1}A^{(i_1,i_2)}\y_j^{(i_2)}\sqrt{t}+(s-1)(C^{(i_1)})^{s-2}\beta_{i_1}\sum_{p_2=1}^{l}A^{(i_1,p_2)}\y_j^{(p_2)}\sqrt{t}.\nonumber \\
\end{eqnarray}
A combination of (\ref{eq:genanal13}), (\ref{eq:genanal17}), (\ref{eq:genanal18}), (\ref{eq:genanal18a}), and (\ref{eq:genanal18b}) gives
\begin{multline}\label{eq:genanal19}
\mE_{\u^{(i_1,1)},\u^{(2)},\u^{(i_1,3)},u^{(4)}}  \frac{(C^{(i_1)})^{s-1} A^{(i_1,i_2)}\u_j^{(i_1,1)}\y_j^{(i_2)}}{Z}  \\
 =
\mE \lp \frac{\y_j^{(i_2)}}{Z}\lp(C^{(i_1)})^{s-1}\beta_{i_1}A^{(i_1,i_2)}\y_j^{(i_2)}\sqrt{t}+(s-1)(C^{(i_1)})^{s-2}\beta_{i_1}\sum_{p_2=1}^{l}A^{(i_1,p_2)}\y_j^{(p_2)}\sqrt{t}\rp \rp \\
-
\mE \lp\sum_{p_1=1}^{l} \frac{(\x^{(i_1)})^T\x^{(p_1)}}{\|\x^{(i_1)}\|_2\|\x^{(p_1)}\|_2}
\frac{(C^{(i_1)})^{s-1} A^{(i_1,i_2)}\y_j^{(i_2)}}{Z^2}
s  (C^{(p_1)})^{s-1}\sum_{p_2=1}^{l}\beta_{p_1}A^{(p_1,p_2)}\y_j^{(p_2)}\sqrt{t}\rp.
\end{multline}

\subsubsection{Finding $\mE_{\u^{(i_1,1)},\u^{(2)},\u^{(i_1,3)},u^{(4)}}  \frac{(C^{(i_1)})^{s-1} A^{(i_1,i_2)}\u_j^{(2)}\y_j^{(i_2)}}{Z}$}
\label{sec:hand2}

We start with the following standard utilization of the Gaussian integration by parts.
\begin{multline}\label{eq:genAanal12}
\mE_{\u^{(i_1,1)},\u^{(2)},\u^{(i_1,3)},u^{(4)}}  \frac{(C^{(i_1)})^{s-1} A^{(i_1,i_2)}\u_j^{(2)}\y_j^{(i_2)}}{Z} =
\mE(\mE (\u_j^{(2)}\u_j^{(2)})\frac{d}{d\u_j^{(2)}}\lp \frac{(C^{(i_1)})^{s-1} A^{(i_1,i_2)}\y_j^{(i_2)}}{Z}\rp).
\end{multline}
Obviously $\mE (\u_j^{(2)}\u_j^{(2)})=1$ and we also have
\begin{eqnarray}\label{eq:genAanal18}
\frac{d}{d\u_j^{(2)}}\lp \frac{(C^{(i_1)})^{s-1} A^{(i_1,i_2)}\y_j^{(i_2)}}{Z}\rp  & = &
\frac{\y_j^{(i_2)}}{Z}\frac{d}{d\u_j^{(2)}}\lp (C^{(i_1)})^{s-1} A^{(i_1,i_2)}\rp-\frac{ (C^{(i_1)})^{s-1} A^{(i_1,i_2)}\y_j^{(i_2)}}{Z^2}
\frac{dZ}{d\u_j^{(2)}}.\nonumber \\
\end{eqnarray}
Moreover, we find
\begin{multline}\label{eq:genAanal18a}
\frac{dZ}{d\u_j^{(2)}}=\frac{d\sum_{i_1=1}^{l}  (C^{(i_1)})^s}{d\u_j^{(2)}}
=s  (C^{(i_1)})^{s-1}\sum_{p_2=1}^{l}\frac{d(A^{(i_1,p_2)})}{d\u_j^{(2)}}=s  (C^{(i_1)})^{s-1}\sum_{p_2=1}^{l}
\beta_{i_1}A^{(i_1,p_2)}\y_j^{(p_2)}\sqrt{1-t}.
\end{multline}
It is not that hard to obtain the following as well
\begin{multline}\label{eq:genAanal18b}
\frac{d}{d\u_j^{(2)}}\lp (C^{(i_1)})^{s-1} A^{(i_1,i_2)}\rp
 =  (C^{(i_1)})^{s-1} \frac{dA^{(i_1,i_2)} }{d\u_j^{(2)}}+ A^{(i_i,i_2)}\frac{d(C^{(i_1)})^{s-1}}{d\u_j^{(2)}} \\
 =  (C^{(i_1)})^{s-1}\beta_{i_1}A^{(i_1,i_2)}\y_j^{(i_2)}\sqrt{1-t}+(s-1)(C^{(i_1)})^{s-2}\beta_{i_1}\sum_{p_2=1}^{l}A^{(i_1,p_2)}\y_j^{(p_2)}\sqrt{1-t}.
\end{multline}
Combining (\ref{eq:genAanal12}), (\ref{eq:genAanal18}), (\ref{eq:genAanal18a}), and (\ref{eq:genAanal18b}) we have
\begin{multline}\label{eq:genAanal19}
\mE_{\u^{(i_1,1)},\u^{(2)},\u^{(i_1,3)},u^{(4)}}  \frac{(C^{(i_1)})^{s-1} A^{(i_1,i_2)}\u_j^{(2)}\y_j^{(i_2)}}{Z}  \\
 =
\mE \lp\frac{\y_j^{(i_2)}}{Z}\lp(C^{(i_1)})^{s-1}\beta_{i_1}A^{(i_1,i_2)}\y_j^{(i_2)}\sqrt{1-t}+(s-1)(C^{(i_1)})^{s-2}\beta_{i_1}\sum_{p_2=1}^{l}A^{(i_1,p_2)}\y_j^{(p_2)}\sqrt{1-t}\rp \rp \\
-
\mE \lp\sum_{p_1=1}^{l}
\frac{(C^{(i_1)})^{s-1} A^{(i_1,i_2)}\y_j^{(i_2)}}{Z^2}
s  (C^{(p_1)})^{s-1}\sum_{p_2=1}^{l}\beta_{p_1}A^{(p_1,p_2)}\y_j^{(p_2)}\sqrt{1-t}\rp.
\end{multline}

\subsubsection{Finding $\mE_{\u^{(i_1,1)},\u^{(2)},\u^{(i_1,3)},u^{(4)}}  \frac{(C^{(i_1)})^{s-1} A^{(i_1,i_2)}\u^{(i_1,3)}}{Z}$}
\label{sec:hand3}

We closely follow what we presented above and start with the following utilization of the Gaussian integration by parts
\begin{multline}\label{eq:genBanal12}
\mE_{\u^{(i_1,1)},\u^{(2)},\u^{(i_1,3)},u^{(4)}}  \frac{(C^{(i_1)})^{s-1} A^{(i_1,i_2)}\u^{(i_1,3)}}{Z}
=\mE(\sum_{p_1=1,p_1\neq i_1}^{l}\mE (\u^{(i_1,3)}\u^{(p_1,3)})\frac{d}{d\u^{(p_1,3)}}\lp \frac{(C^{(i_1)})^{s-1} A^{(i_1,i_2)}}{Z}\rp)\\
+\mE(\mE (\u^{(i_1,3)}\u^{(i_1,3)})\frac{d}{d\u^{(i_1,3)}}\lp \frac{(C^{(i_1)})^{s-1} A^{(i_1,i_2)}}{Z}\rp).
\end{multline}
Clearly, $\mE (\u^{(i_1,3)}\u^{(p_1,3)})=\frac{(\x^{(i_1)})^T\x^{(p_1)}}{\|\x^{(i_1)}\|_2\|\x^{(p_1)}\|_2}$ and for $p_1\neq i_1$ we obtain the following
\begin{equation}\label{eq:genBanal14}
\frac{d}{d\u^{(p_1,3)}}\lp \frac{(C^{(i_1)})^{s-1} A^{(i_1,i_2)}}{Z}\rp=(C^{(i_1)})^{s-1} A^{(i_1,i_2)}\frac{d}{d\u^{(p_1,3)}}\lp \frac{1}{Z}\rp=-\frac{(C^{(i_1)})^{s-1} A^{(i_1,i_2)}}{Z^2}\frac{dZ}{d\u^{(p_1,3)}}.
\end{equation}
Following (\ref{eq:genBanal14a}) we also have
\begin{equation}\label{eq:genBanal14a}
\frac{dZ}{d\u^{(p_1,3)}}=\frac{d\sum_{i_1=1}^{l}  (C^{(i_1)})^s}{d\u^{(p_1,3)}}
=s  (C^{(p_1)})^{s-1}\sum_{p_2=1}^{l}\frac{d(A^{(p_1,p_2)})}{d\u^{(p_1,3)}}.
\end{equation}
From (\ref{eq:genanal7}) we find
\begin{equation}\label{eq:genBanal16}
\frac{d(A^{(p_1,p_2)})}{d\u^{(p_1,3)}}
=\beta_{p_1}A^{(p_1,p_2)}\|\y^{(p_2)}\|_2\sqrt{1-t}.
\end{equation}
Combining (\ref{eq:genBanal14}), (\ref{eq:genBanal14a}), and (\ref{eq:genBanal16}) we obtain
\begin{eqnarray}\label{eq:genBanal17}
\frac{d}{d\u^{(p_1,3)}}\lp \frac{(C^{(i_1)})^{s-1} A^{(i_1,i_2)}}{Z}\rp & = & -\frac{(C^{(i_1)})^{s-1} A^{(i_1,i_2)}}{Z^2}
s  (C^{(p_1)})^{s-1}\sum_{p_2=1}^{l}\frac{d(A^{(p_1,p_2)})}{d\u^{(p_1,3)}}\nonumber \\
 & = & -\frac{(C^{(i_1)})^{s-1} A^{(i_1,i_2)}}{Z^2}
s  (C^{(p_1)})^{s-1}\sum_{p_2=1}^{l}\beta_{p_1}A^{(p_1,p_2)}\|\y^{(p_2)}\|_2\sqrt{1-t}.
\end{eqnarray}
Also, we easily have $\mE (\u^{(i_1,3)}\u^{(i_1,3)})=1$ and
\begin{eqnarray}\label{eq:genBanal18}
\frac{d}{d\u^{(i_1,3)}}\lp \frac{(C^{(i_1)})^{s-1} A^{(i_1,i_2)}}{Z}\rp  & = &
\frac{1}{Z}\frac{d}{d\u^{(i_1,3)}}\lp (C^{(i_1)})^{s-1} A^{(i_1,i_2)}\rp-\frac{ (C^{(i_1)})^{s-1} A^{(i_1,i_2)}}{Z^2}
\frac{dZ}{d\u^{(i_1,3)}}.\nonumber \\
\end{eqnarray}
Moreover,
\begin{multline}\label{eq:genBanal18a}
\frac{dZ}{d\u^{(i_1,3)}}=\frac{d\sum_{p_1=1}^{l}  (C^{(p_1)})^s}{d\u^{(i_1,3)}}
=s  (C^{(i_1)})^{s-1}\sum_{p_2=1}^{l}\frac{d(A^{(i_1,p_2)})}{d\u^{(i_1,3)}}=s  (C^{(i_1)})^{s-1}\sum_{p_2=1}^{l}
\beta_{i_1}A^{(i_1,p_2)}\|\y^{(p_2)}\|_2\sqrt{1-t}.
\end{multline}
Similarly to what was done in (\ref{eq:genanal18b}) we find
\begin{eqnarray}\label{eq:genBanal18b}
\frac{d}{d\u^{(i_1,3)}}\lp (C^{(i_1)})^{s-1} A^{(i_1,i_2)}\rp
& = & (C^{(i_1)})^{s-1} \frac{dA^{(i_1,i_2)} }{d\u^{(i_1,3)}}+ A^{(i_i,i_2)}\frac{d(C^{(i_1)})^{s-1}}{d\u^{(i_1,3)}} \nonumber \\
& = & (C^{(i_1)})^{s-1}\beta_{i_1}A^{(i_1,i_2)}\|\y^{(i_2)}\|_2\sqrt{1-t}\nonumber \\
& & +(s-1)(C^{(i_1)})^{s-2}\beta_{i_1}\sum_{p_2=1}^{l}A^{(i_1,p_2)}\|\y^{(p_2)}\|_2\sqrt{1-t}.\nonumber \\
\end{eqnarray}
Combining (\ref{eq:genBanal12}), (\ref{eq:genBanal18}), (\ref{eq:genBanal18a}), and (\ref{eq:genBanal18b}) we find
\begin{multline}\label{eq:genBanal19}
\mE_{\u^{(i_1,1)},\u^{(2)},\u^{(i_1,3)},u^{(4)}}  \frac{(C^{(i_1)})^{s-1} A^{(i_1,i_2)}\u^{(i_1,3)}}{Z}  \\
 =
\mE \lp\frac{1}{Z}\lp(C^{(i_1)})^{s-1}\beta_{i_1}A^{(i_1,i_2)}\|\y^{(i_2)}\|_2\sqrt{1-t}+(s-1)(C^{(i_1)})^{s-2}\beta_{i_1}\sum_{p_2=1}^{l}A^{(i_1,p_2)}\|\y^{(p_2)}\|_2\sqrt{1-t}\rp \rp \\
-
\mE \lp\sum_{p_1=1}^{l}\frac{(\x^{(i_1)})^T\x^{(p_1)}}{\|\x^{(i_1)}\|_2\|\x^{(p_1)}\|_2}
\frac{(C^{(i_1)})^{s-1} A^{(i_1,i_2)}}{Z^2}
s  (C^{(p_1)})^{s-1}\sum_{p_2=1}^{l}\beta_{p_1}A^{(p_1,p_2)}\|\y^{(p_2)}\|_2\sqrt{1-t}\rp.
\end{multline}

\subsubsection{Finding $\mE_{\u^{(i_1,1)},\u^{(2)},\u^{(i_1,3)},u^{(4)}}  \frac{(C^{(i_1)})^{s-1} A^{(i_1,i_2)}u^{(4)}}{Z}$}
\label{sec:hand4}

We again closely follow what we presented above and start with the following utilization of the Gaussian integration by parts
\begin{equation}\label{eq:genCanal12}
\mE_{\u^{(i_1,1)},\u^{(2)},\u^{(i_1,3)},u^{(4)}}  \frac{(C^{(i_1)})^{s-1} A^{(i_1,i_2)}u^{(4)}}{Z}
=\mE(\mE (u^{(4)}u^{(4)})\frac{d}{du^{(4)}}\lp \frac{(C^{(i_1)})^{s-1} A^{(i_1,i_2)}}{Z}\rp).
\end{equation}
Clearly, $\mE (u^{(4)}u^{(4)})=1$. Further, we have
\begin{eqnarray}\label{eq:genCanal18}
\frac{d}{du^{(4)}}\lp \frac{(C^{(i_1)})^{s-1} A^{(i_1,i_2)}}{Z}\rp  & = &
\frac{1}{Z}\frac{d}{d u^{(4)}}\lp (C^{(i_1)})^{s-1} A^{(i_1,i_2)}\rp-\frac{ (C^{(i_1)})^{s-1} A^{(i_1,i_2)}}{Z^2}
\frac{dZ}{d u^{(4)}}.\nonumber \\
\end{eqnarray}
Similarly to (\ref{eq:genBanal18a}) we find
\begin{equation}\label{eq:genCanal18a}
\frac{dZ}{du^{(4)}}=\frac{d\sum_{p_1=1}^{l}  (C^{(p_1)})^s}{du^{(4)}}
=s \sum_{p_1=1}^{l} (C^{(p_1)})^{s-1}\sum_{p_2=1}^{l}\frac{d(A^{(p_1,p_2)})}{du^{(4)}}=s \sum_{p_1=1}^{l} (C^{(p_1)})^{s-1}\sum_{p_2=1}^{l}
\beta_{p_1}A^{(p_1,p_2)}\|\y^{(p_2)}\|_2\sqrt{t}.
\end{equation}
Following closely (\ref{eq:genBanal18b}) (and earlier (\ref{eq:genanal18b})) we also find
\begin{multline}\label{eq:genCanal18b}
\frac{d}{du^{(4)}}\lp (C^{(i_1)})^{s-1} A^{(i_1,i_2)}\rp
 =  (C^{(i_1)})^{s-1} \frac{dA^{(i_1,i_2)} }{du^{(4)}}+ A^{(i_i,i_2)}\frac{d(C^{(i_1)})^{s-1}}{d u^{(4)}} \\
 =  (C^{(i_1)})^{s-1}\beta_{i_1}A^{(i_1,i_2)}\|\y^{(i_2)}\|_2\sqrt{t}
 +(s-1)(C^{(i_1)})^{s-2}\beta_{i_1}\sum_{p_2=1}^{l}A^{(i_1,p_2)}\|\y^{(p_2)}\|_2\sqrt{t}.
\end{multline}
A combination of (\ref{eq:genCanal12}), (\ref{eq:genCanal18}), (\ref{eq:genCanal18a}), and (\ref{eq:genCanal18b}) gives
\begin{multline}\label{eq:genCanal19}
\mE_{\u^{(i_1,1)},\u^{(2)},\u^{(i_1,3)},u^{(4)}}  \frac{(C^{(i_1)})^{s-1} A^{(i_1,i_2)}u^{(4)}}{Z}  \\
 =
\mE \lp\frac{1}{Z}\lp(C^{(i_1)})^{s-1}\beta_{i_1}A^{(i_1,i_2)}\|\y^{(i_2)}\|_2\sqrt{t}+(s-1)(C^{(i_1)})^{s-2}\beta_{i_1}\sum_{p_2=1}^{l}A^{(i_1,p_2)}\|\y^{(p_2)}\|_2\sqrt{t}\rp \rp \\
-
\mE \lp
\frac{(C^{(i_1)})^{s-1} A^{(i_1,i_2)}}{Z^2}
s  \sum_{p_1=1}^{l} (C^{(p_1)})^{s-1}\sum_{p_2=1}^{l}\beta_{p_1}A^{(p_1,p_2)}\|\y^{(p_2)}\|_2\sqrt{t}\rp.
\end{multline}

\subsubsection{Connecting all pieces together}
\label{sec:conalt}

Using (\ref{eq:genanal11}), (\ref{eq:genanal19}), (\ref{eq:genAanal19}), (\ref{eq:genBanal19}), and (\ref{eq:genCanal19}) we obtain
\begin{eqnarray}\label{eq:conalt1}
\frac{\psi(\calX,\calY,\beta,s,t)}{dt}
 =  \frac{s}{2\beta|s|\sqrt{n}} \mE_{\u^{(i_1,1)},\u^{(2)},\u^{(i_1,3)},u^{(4)}} (-S_1+S_2+S_3-S_4)
\end{eqnarray}
where
\begin{eqnarray}
\label{eq:conalt1a}
S1 & = & \sum_{i_1=1}^{l} \sum_{i_2=1}^{l}\beta_{i_1}\sum_{j=1}^{m}
\lp\sum_{p_1=1}^{l} \frac{(\x^{(i_1)})^T\x^{(p_1)}}{\|\x^{(i_1)}\|_2\|\x^{(p_1)}\|_2}
\frac{(C^{(i_1)})^{s-1} A^{(i_1,i_2)}\y_j^{(i_2)}}{Z^2}
s  (C^{(p_1)})^{s-1}\sum_{p_2=1}^{l}\beta_{p_1}A^{(p_1,p_2)}\y_j^{(p_2)}\rp\nonumber \\
S_2 & = & \sum_{i_1=1}^{l} \sum_{i_2=1}^{l}\beta_{i_1}\sum_{j=1}^{m}
 \lp\sum_{p_1=1}^{l}
\frac{(C^{(i_1)})^{s-1} A^{(i_1,i_2)}\y_j^{(i_2)}}{Z^2}
s  (C^{(p_1)})^{s-1}\sum_{p_2=1}^{l}\beta_{p_1}A^{(p_1,p_2)}\y_j^{(p_2)}\rp\nonumber \\
S_3 & = & \sum_{i_1=1}^{l} \sum_{i_2=1}^{l}\beta_{i_1}
\|\y^{(i_2)}\|_2 \lp\sum_{p_1=1}^{l}\frac{(\x^{(i_1)})^T\x^{(p_1)}}{\|\x^{(i_1)}\|_2\|\x^{(p_1)}\|_2}
\frac{(C^{(i_1)})^{s-1} A^{(i_1,i_2)}}{Z^2}
s  (C^{(p_1)})^{s-1}\sum_{p_2=1}^{l}\beta_{p_1}A^{(p_1,p_2)}\|\y^{(p_2)}\|_2\rp\nonumber \\
S_4 & = & \sum_{i_1=1}^{l} \sum_{i_2=1}^{l}\beta_{i_1}
\|\y^{(i_2)}\|_2 \lp
\frac{(C^{(i_1)})^{s-1} A^{(i_1,i_2)}}{Z^2}
s \sum_{p_1=1}^{l}  (C^{(p_1)})^{s-1}\sum_{p_2=1}^{l}\beta_{p_1}A^{(p_1,p_2)}\|\y^{(p_2)}\|_2\rp.
\end{eqnarray}
From (\ref{eq:conalt1a}) we further have
\begin{multline}
\label{eq:conalt1b}
S_2-S_1
 =  s\beta^2 \sum_{i_1=1}^{l} \sum_{p_1=1}^{l}
\frac{(C^{(i_1)})^{s}(C^{(p_1)})^{s}(\|\x^{(i_1)}\|_2\|\x^{(p_1)}\|_2-(\x^{(i_1)})^T\x^{(p_1)})}{Z^2} \\
\times \lp\sum_{i_2=1}^{l}\sum_{p_2=1}^{l}
\frac{A^{(i_1,i_2)}A^{(p_1,p_2)}}{C^{(i_1)}C^{(p_1)}}
(\y^{(i_2)})^T\y^{(p_2)}\rp,
\end{multline}
and in a similar fashion
\begin{multline}
\label{eq:conalt1c}
S_4-S_3
 =  s\beta^2 \sum_{i_1=1}^{l} \sum_{p_1=1}^{l}
\frac{(C^{(i_1)})^{s}(C^{(p_1)})^{s}(\|\x^{(i_1)}\|_2\|\x^{(p_1)}\|_2-(\x^{(i_1)})^T\x^{(p_1)})}{Z^2} \\
\times \lp\sum_{i_2=1}^{l}\sum_{p_2=1}^{l} \frac{A^{(i_1,i_2)}A^{(p_1,p_2)}}{C^{(i_1)}C^{(p_1)}}
\|\y^{(i_2)})\|_2\|\y^{(p_2)}\|_2\rp.
\end{multline}
Combining (\ref{eq:conalt1a}), (\ref{eq:conalt1b}), and (\ref{eq:conalt1c}) we finally have
\begin{multline}\label{eq:conalt2}
\frac{\psi(\calX,\calY,\beta,s,t)}{dt}
  =   -\frac{s^2\beta}{2|s|\sqrt{n}} \mE_{\u^{(i_1,1)},\u^{(2)},\u^{(i_1,3)},u^{(4)}}  \sum_{i_1=1}^{l} \sum_{p_1=1}^{l}
\frac{(C^{(i_1)})^{s}(C^{(p_1)})^{s}(\|\x^{(i_1)}\|_2\|\x^{(p_1)}\|_2-(\x^{(i_1)})^T\x^{(p_1)})}{Z^2}  \\
 \times \lp\sum_{i_2=1}^{l}\sum_{p_2=1}^{l}
\frac{A^{(i_1,i_2)}A^{(p_1,p_2)}}{C^{(i_1)}C^{(p_1)}}
(\|\y^{(i_2)}\|_2\|\y^{(p_2)}\|_2-(\y^{(i_2)})^T\y^{(p_2)})\rp.
\end{multline}
 Now it easily follows that $\frac{\psi(\calX,\calY,\beta,s,t)}{dt}\leq 0$ and function $\psi(\calX,\calY,\beta,s,t)$ is indeed non-increasing (decreasing) in $t$. We summarize the obtained results in the following theorem.
\begin{theorem}
\label{thm:thm1}
Let $G\in\mR^{m \times n},u^{(4)}\in\mR^1,\u^{(2)}\in\mR^{m\times 1}$, and $\h\in\mR^{n\times 1}$ all have i.i.d. standard normal components ($G$, $u^{(4)}$, $\u^{(2)}$, and $\h$ are then independent of each other as well). Assume that set $\calX=\{\x^{(1)},\x^{(2)},\dots,\x^{(l)}\}$, where $\x^{(i)}\in \mR^{n},1\leq i\leq l$, and set $\calY=\{\y^{(1)},\y^{(2)},\dots,\y^{(l)}\}$, where $\y^{(i)}\in \mR^{m},1\leq i\leq l$ are given and that $\beta\geq 0$ and $s$ are real numbers. One then has that function $\psi(\calX,\calY,\beta,s,t)$
\begin{multline}\label{eq:thm1eq1}
\psi(\calX,\calY,\beta,s,t)=  \mE_{G,u^{(4)},\u^{(2)},\h} \frac{1}{\beta|s|\sqrt{n}} \\
 \times \log\lp \sum_{i_1=1}^{l}\lp\sum_{i_2=1}^{l}e^{\beta \lp \sqrt{t}(\y^{(i_2)})^T
 G\x^{(i_1)}+\sqrt{1-t}\|\x^{(i_2)}\|_2 (\y^{(i_2)})^T\u^{(2)}+\sqrt{t}\|\x^{(i_1)}\|_2\|\y^{(i_2)}\|_2,u^{(4)} +\sqrt{1-t}\|\y^{(i_2)}\|_2\h^T\x^{(i)}\rp} \rp^{s}\rp,
\end{multline}
is non-increasing (decreasing) in $t$.
\end{theorem}
\begin{proof}
  Follows from the above presentation.
\end{proof}
\begin{corollary}
  Assume the setup of Theorem \ref{thm:thm1}. Then we also have
\begin{eqnarray}\label{eq:co1eq1}
\psi(\calX,\calY,\beta,s,t)= \psi(\calX,\calY,\beta,s,0)+\int_{0}^{t}\frac{d\psi(\calX,\calY,\beta,s,t)}{dt}dt,
\end{eqnarray}
as well as the following comparison principle
\begin{eqnarray}\label{eq:co1eq2}
\psi(\calX,\calY,\beta,s,0) \geq  \psi(\calX,\calY,\beta,s,t)\geq \psi(\calX,\calY,\beta,s,1).
\end{eqnarray}
\end{corollary}
\begin{proof}
It is automatic by the above theorem and after one notes that $\frac{\psi(\calX,\calY,\beta,s,t)}{dt}\leq 0$.
\end{proof}

\subsection{Numerical experiments}
\label{sec:genconsim}

The theoretical results that we presented above establish a very powerful tool for dealing with random processes. Below we look at them from a numerical point of view, i.e. through  numerical simulations. For the simplicity we chose $m=5$, $n=5$, $l=10$, and selected set $\calX$ as the columns of the following matrix (basically $\calX$ was selected the same way as the corresponding set in \cite{Stojnicgscomp16})
\begin{equation}
X^{+}=\begin{bmatrix}
-0.7998 & 0.1004 & -0.7599 & 0.6616 & 0.5864 & -0.4010 & -0.0148 & -0.8320 & 0.3187 & -0.4861 \\
0.1760 & 0.0704 & 0.1056 & -0.1369 & -0.6259 & -0.5289 & -0.3740 & 0.3140 & 0.6299 & -0.5494 \\
0.0806 & -0.9085 & -0.3381 & -0.1970 & -0.1438 & 0.4863 & 0.5832 & 0.0840 & -0.2299 & -0.2647 \\
0.5487 & -0.3120 & -0.5447 & 0.5673 & 0.4870 & -0.5239 & 0.0407 & -0.2955 & 0.3913 & 0.5113 \\
-0.1476 & 0.2497 & -0.0208 & 0.4276 & 0.0808 & -0.2202 & -0.7198 & 0.3389 & 0.5438 & -0.3611
\end{bmatrix}.
\end{equation}
One then obviously has
\begin{equation}\label{eq:sim1}
  \calX^{+}=\{X^{+}_{:,1},X^{+}_{:,2},\dots,X^{+}_{:,l}\}.
\end{equation}
We do recall the observation from \cite{Stojnicgscomp16} that set $\calX^{+}$ (and matrix $X^{+}$) are practically randomly chosen (an added scaling makes $\|X^{+}_{:,i_1}\|_2=1,1\leq i_1\leq l$). Also we selected set $\calY$ as the columns of the following matrix
\begin{equation}
Y^{+}=\begin{bmatrix}
-0.4639 & 0.7324 & -0.4828 & 0.0280 & -0.4016 & -0.6764 & 0.6161 & 0.4281 & -0.3831 & 0.0699 \\
0.0416 & -0.3678 & 0.0144 & -0.4856 & 0.4880 & -0.6861 & 0.1266 & 0.5132 & 0.0350 & -0.0308 \\
-0.6522 & 0.1775 & 0.2449 & -0.2417 & -0.1255 & 0.2355 & 0.0859 & -0.1498 & 0.2410 & -0.7208 \\
-0.5981 & -0.1078 & 0.4879 & -0.3456 & 0.5796 & -0.0856 & 0.6892 & 0.1325 & 0.8628 & -0.1637 \\
-0.0037 & 0.5340 & 0.6846 & 0.7652 & -0.4989 & -0.0946 & -0.3492 & -0.7165 & -0.2225 & -0.6692 
\end{bmatrix}.
\end{equation}
Clearly,
\begin{equation}\label{eq:sim1a}
  \calY^{+}=\{Y^{+}_{:,1},Y^{+}_{:,2},\dots,Y^{+}_{:,l}\}.
\end{equation}
Similarly to what was mentioned above for set $\calX^{+}$, we also add that set $\calY^{+}$ (and matrix $Y^{+}$) are again for all practical purposes randomly chosen (to make everything a bit neater we again scaled all the columns of $Y^{+}$ so that $\|Y^{+}_{:,i_1}\|_2=1,1\leq i_1\leq l$). The numerical experiments were conducted in a fashion very similar to the one from \cite{Stojnicgscomp16}.
Namely, we simulated derivatives $\frac{d\psi(\calX,\calY,\beta,s,t)}{dt}$ using both (\ref{eq:genanal11}) and (\ref{eq:conalt2}).
We refer to the use of (\ref{eq:genanal11}) as the standard interpolation and to the use of (\ref{eq:conalt2}) as the computed interpolation. We then computed $\psi(\calX,\calY,\beta,s,t)$ using (\ref{eq:co1eq1}). Moreover, we additionally simulated $\psi(\calX,\calY,\beta,s,t)$ using (\ref{eq:genanal8}) which we view as a direct way of simulation without any interpolating computations. We set $\beta=10$ and averaged all random quantities over a set of $5e4$ experiments. To parallel the presentation given in \cite{Stojnicgscomp16} as much as possible, we here also simulated two different scenarios with all other parameters being the same, except that in one of the scenarios $s=1$ and in the other $s=-1$.

\textbf{\underline{\emph{1) $s=1$ -- numerical results}}}

Figure \ref{fig:gensplus1xnorm1psi} and Table \ref{tab:gensplus1xnorm1psi} contain the results obtained for $s=1$. Following the standard that we set in \cite{Stojnicgscomp16}, Figure \ref{fig:gensplus1xnorm1psi} shows the entire range for $t$ (i.e. its shows the values for $t\in(0,1)$) whereas Table \ref{tab:gensplus1xnorm1psi} focuses on several particular values of $t$ and shows concrete values of all key quantities. As both, Figure \ref{fig:gensplus1xnorm1psi} and Table \ref{tab:gensplus1xnorm1psi}, show, there is a solid agreement between all presented results.

\begin{figure}[htb]
\centering
\centerline{\epsfig{figure=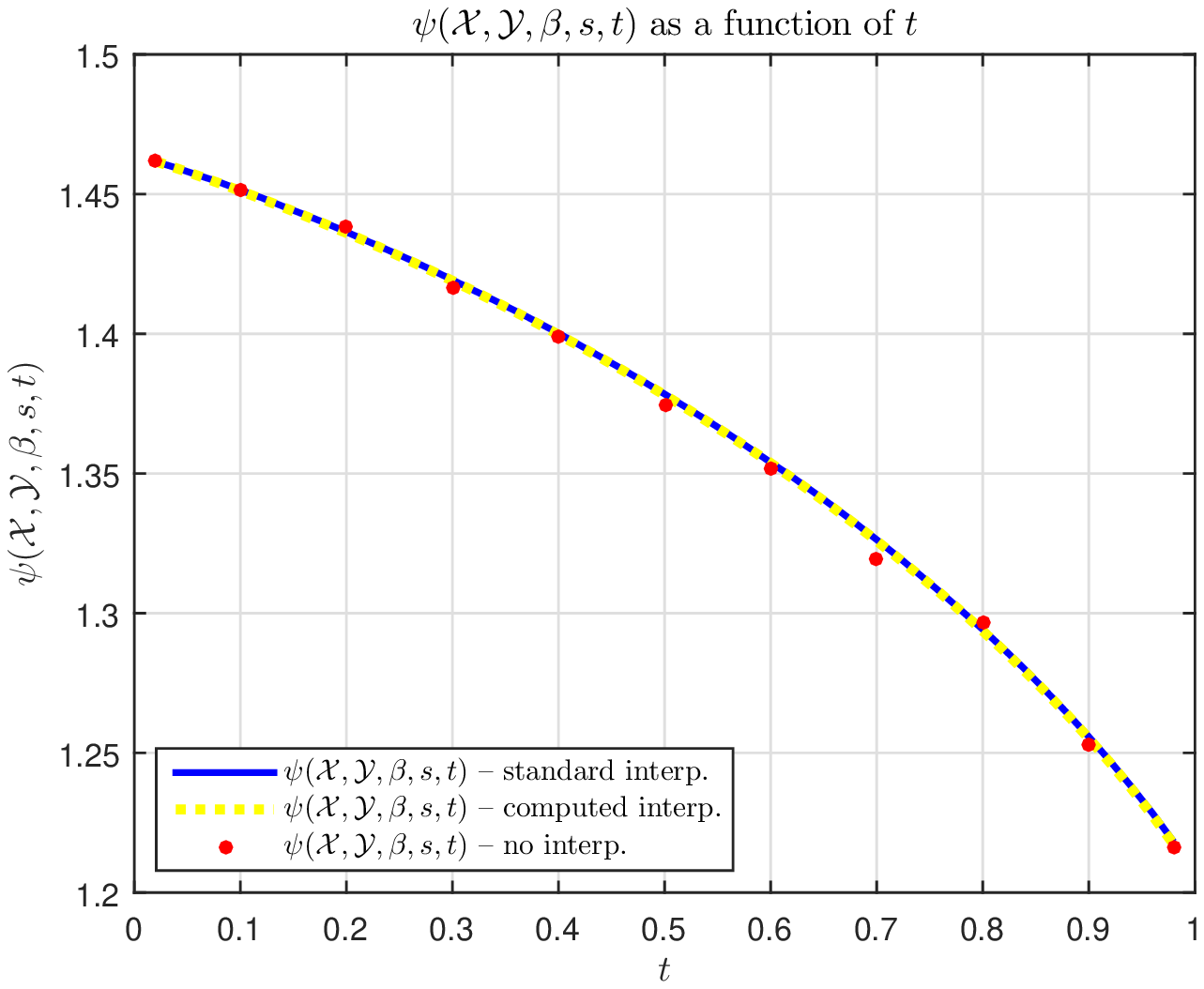,width=11.5cm,height=8cm}}
\caption{$\psi(\calX,\calY,\beta,s,t)$ as a function of $t$; $m=5$, $n=5$, $l=10$, $\calX=\calX^{+}$, $\calY=\calY^{+}$, $\beta=3$, $s=1$}
\label{fig:gensplus1xnorm1psi}
\end{figure}
\begin{table}[h]
\caption{Simulated results --- $m=5$, $n=5$, $l=10$, $\calX=\calX^{+}$, $\calY=\calY^{+}$, $\beta=3$, $s=1$}\vspace{.1in}
\hspace{-0in}\centering
\begin{tabular}{||c||c|c|c|c|c||}\hline\hline
$ t$  &  $\frac{d\psi}{dt}$; (\ref{eq:genanal11}) & $\frac{d\psi}{dt}$;  (\ref{eq:conalt2}) & $\psi$;  (\ref{eq:genanal11}) and (\ref{eq:co1eq1}) & $\psi$; (\ref{eq:conalt2}) and (\ref{eq:co1eq1}) & $\psi$;  (\ref{eq:genanal8})\\  \hline\hline
$ 0.1000 $ & $ -0.1438 $ & $ -0.1384 $ & $\bl{\mathbf{ 1.4514 }}$ & $\bl{\mathbf{ 1.4511 }}$ & $\mathbf{ 1.4514 }$ \\ \hline
$ 0.2000 $ & $ -0.1613 $ & $ -0.1574 $ & $\bl{\mathbf{ 1.4365 }}$ & $\bl{\mathbf{ 1.4361 }}$ & $\mathbf{ 1.4379 }$ \\ \hline
$ 0.3000 $ & $ -0.1819 $ & $ -0.1794 $ & $\bl{\mathbf{ 1.4193 }}$ & $\bl{\mathbf{ 1.4190 }}$ & $\mathbf{ 1.4162 }$ \\ \hline
$ 0.4000 $ & $ -0.2003 $ & $ -0.2019 $ & $\bl{\mathbf{ 1.4002 }}$ & $\bl{\mathbf{ 1.3997 }}$ & $\mathbf{ 1.3988 }$ \\ \hline
$ 0.5000 $ & $ -0.2252 $ & $ -0.2269 $ & $\bl{\mathbf{ 1.3784 }}$ & $\bl{\mathbf{ 1.3781 }}$ & $\mathbf{ 1.3746 }$ \\ \hline
$ 0.6000 $ & $ -0.2569 $ & $ -0.2554 $ & $\bl{\mathbf{ 1.3540 }}$ & $\bl{\mathbf{ 1.3537 }}$ & $\mathbf{ 1.3518 }$ \\ \hline
$ 0.7000 $ & $ -0.2957 $ & $ -0.2934 $ & $\bl{\mathbf{ 1.3263 }}$ & $\bl{\mathbf{ 1.3259 }}$ & $\mathbf{ 1.3192 }$ \\ \hline
$ 0.8000 $ & $ -0.3359 $ & $ -0.3452 $ & $\bl{\mathbf{ 1.2942 }}$ & $\bl{\mathbf{ 1.2936 }}$ & $\mathbf{ 1.2964 }$ \\ \hline
$ 0.9000 $ & $ -0.4137 $ & $ -0.4164 $ & $\bl{\mathbf{ 1.2558 }}$ & $\bl{\mathbf{ 1.2552 }}$ & $\mathbf{ 1.2531 }$ \\ \hline \hline
\end{tabular}
\label{tab:gensplus1xnorm1psi}
\end{table}

\vspace{.2in}
\textbf{\underline{\emph{2) $s=-1$ -- numerical results}}}

Figure \ref{fig:gensmin1xnorm1psi} and Table \ref{tab:gensmin1xnorm1psi} contain the results obtained for $s=-1$. Figure \ref{fig:gensmin1xnorm1psi} again shows the entire range for $t$, whereas Table \ref{tab:gensmin1xnorm1psi} focuses on several particular values of $t$. Similarly to what we had above for $s=1$, here we again have that both, Figure \ref{fig:gensmin1xnorm1psi} and Table \ref{tab:gensmin1xnorm1psi}, show that the agreement between all presented results is fairly strong.

\begin{figure}[htb]
\centering
\centerline{\epsfig{figure=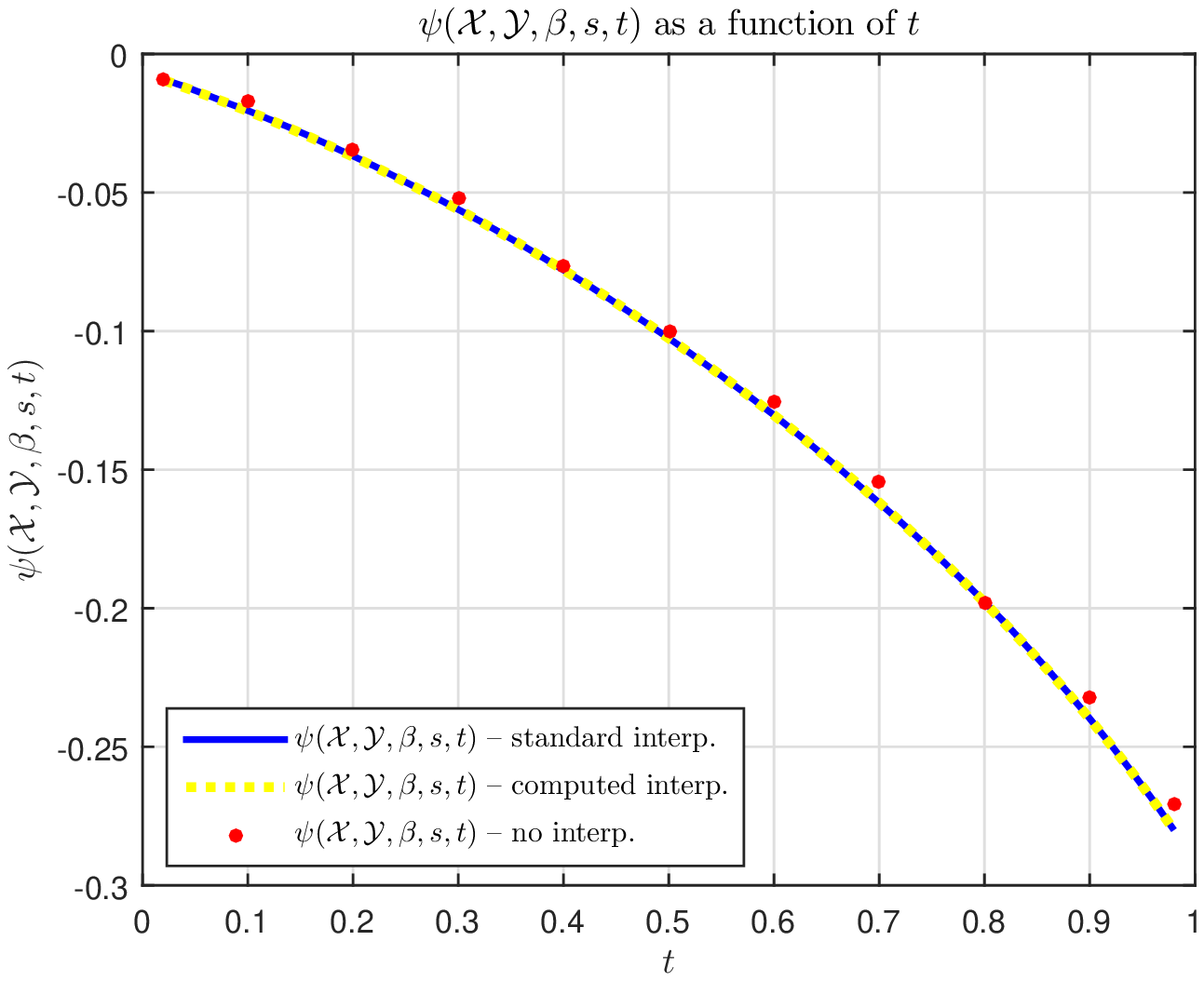,width=11.5cm,height=8cm}}
\caption{$\psi(\calX,\calY,\beta,s,t)$ as a function of $t$; $m=5$, $n=5$, $l=10$, $\calX=\calX^{+}$, $\calY=\calY^{+}$, $\beta=3$, $s=-1$}
\label{fig:gensmin1xnorm1psi}
\end{figure}
\begin{table}[h]
\caption{Simulated results --- $m=5$, $n=5$, $l=10$, $\calX=\calX^{+}$, $\beta=3$, $s=-1$}\vspace{.1in}
\hspace{-0in}\centering
\begin{tabular}{||c||c|c|c|c|c||}\hline\hline
$ t$  &  $\frac{d\psi}{dt}$; (\ref{eq:genanal11}) & $\frac{d\psi}{dt}$;  (\ref{eq:conalt2}) & $\psi$;  (\ref{eq:genanal11}) and (\ref{eq:co1eq1}) & $\psi$; (\ref{eq:conalt2}) and (\ref{eq:co1eq1}) & $\psi$;  (\ref{eq:genanal8})\\  \hline\hline
$ 0.1000 $ & $ -0.1422 $ & $ -0.1471 $ & $\bl{\mathbf{ -0.0204 }}$ & $\bl{\mathbf{ -0.0206 }}$ & $\mathbf{ -0.0172 }$\\ \hline
$ 0.2000 $ & $ -0.1726 $ & $ -0.1735 $ & $\bl{\mathbf{ -0.0368 }}$ & $\bl{\mathbf{ -0.0370 }}$ & $\mathbf{ -0.0349 }$\\ \hline
$ 0.3000 $ & $ -0.2026 $ & $ -0.2010 $ & $\bl{\mathbf{ -0.0561 }}$ & $\bl{\mathbf{ -0.0561 }}$ & $\mathbf{ -0.0518 }$\\ \hline
$ 0.4000 $ & $ -0.2291 $ & $ -0.2296 $ & $\bl{\mathbf{ -0.0780 }}$ & $\bl{\mathbf{ -0.0780 }}$ & $\mathbf{ -0.0769 }$\\ \hline
$ 0.5000 $ & $ -0.2555 $ & $ -0.2594 $ & $\bl{\mathbf{ -0.1026 }}$ & $\bl{\mathbf{ -0.1026 }}$ & $\mathbf{ -0.1000 }$\\ \hline
$ 0.6000 $ & $ -0.2889 $ & $ -0.2923 $ & $\bl{\mathbf{ -0.1303 }}$ & $\bl{\mathbf{ -0.1304 }}$ & $\mathbf{ -0.1254 }$\\ \hline
$ 0.7000 $ & $ -0.3269 $ & $ -0.3331 $ & $\bl{\mathbf{ -0.1620 }}$ & $\bl{\mathbf{ -0.1619 }}$ & $\mathbf{ -0.1546 }$\\ \hline
$ 0.8000 $ & $ -0.3769 $ & $ -0.3818 $ & $\bl{\mathbf{ -0.1977 }}$ & $\bl{\mathbf{ -0.1979 }}$ & $\mathbf{ -0.1981 }$\\ \hline
$ 0.9000 $ & $ -0.4533 $ & $ -0.4503 $ & $\bl{\mathbf{ -0.2398 }}$ & $\bl{\mathbf{ -0.2400 }}$ & $\mathbf{ -0.2324 }$
\\ \hline \hline
\end{tabular}
\label{tab:gensmin1xnorm1psi}
\end{table}

\subsection{$\beta\rightarrow \infty$}
\label{sec:betainf}

In \cite{Stojnicgscomp16}, we showed that the comparison concepts introduced there in $\beta\rightarrow\infty$ regime simplify to well known forms of Slepian's max and Gordon's minmax principles. Below we show that the comparison principles introduced above behave so to say in a similar way and also contain as a special case (obtained again in $\beta\rightarrow\infty$ regime) both, Slepian's max and Gordon's minmax principles (this time though, the resulting forms are more general). Now, we easily have for the limiting behavior of $\xi(\calX,\calY,\beta,s)$
\begin{eqnarray}\label{eq:betainf1}
\lim_{\beta\rightarrow\infty} \xi(\calX,\calY,\beta,s) & = & \lim_{\beta\rightarrow\infty} \mE_{G,u^{(4)}} \frac{1}{|s|\beta\sqrt{n}} \log\lp\sum_{i=1_2}^{l}\lp \sum_{i=1_2}^{l}e^{\beta\lp (\y^{(i_2)})^T
 G\x^{(i_1)}+\|\x^{(i_1)}\|_2\|\y^{(i_2)}\|_2u^{(4)}\rp} \rp^s\rp \nonumber \\
 & = &
 \lim_{\beta\rightarrow\infty} \mE_{G,u^{(4)}} \frac{1}{|s|\beta\sqrt{n}} \log\lp\sum_{i=1_2}^{l}\lp e^{\beta \max_{\y^{(i_2)}\in \calY}\lp (\y^{(i_2)})^T
 G\x^{(i_1)}+\|\x^{(i_1)}\|_2\|\y^{(i_2)}\|_2u^{(4)}\rp} \rp^s\rp \nonumber \\
 & = &
 \lim_{\beta\rightarrow\infty} \mE_{G,u^{(4)}} \frac{1}{|s|\beta\sqrt{n}} \log\lp e^{\max_{\x^{(i_2)}\in \calX}s\beta \max_{\y^{(i_2)}\in \calY}\lp (\y^{(i_2)})^T
 G\x^{(i_1)}+\|\x^{(i_1)}\|_2\|\y^{(i_2)}\|_2u^{(4)}\rp}\rp \nonumber \\
 & = & \mE_{G,u^{(4)}} \frac{\max_{\x^{(i_1)}\in \calX} \lp \mbox{sign}(s) \max_{\y^{(i)}\in \calY}\lp(\y^{(i_2)})^T
 G\x^{(i)}+\|\x^{(i_1)}\|_2\|\y^{(i_2)}\|_2u^{(4)}\rp\rp}{\sqrt{n}}.
\end{eqnarray}

\subsubsection{$s>0$ -- reestablishing a Slepian's max comparison}
\label{sec:betainfsplus1}

If $s>0$ then (\ref{eq:betainf1}) gives
\begin{eqnarray}\label{eq:betainfsplus1}
\lim_{\beta\rightarrow\infty} \xi(\calX,\beta,s,1)=\mE_{G,u^{(4)}} \frac{\max_{\x^{(i_1)}\in \calX,\y^{(i_2)}\in \calY}\lp(\y^{(i_2)})^T
 G\x^{(i_1)}+\|\x^{(i_1)}\|_2\|\y^{(i_2)}\|_2u^{(4)}\rp}{\sqrt{n}}.
\end{eqnarray}
We now recall that $\xi(\calX,\calY,\beta,s)=\psi(\calX,\calY,\beta,s,1)$ and utilize the above machinery to find
\begin{multline}\label{eq:betainfsplus2}
\mE_{G,u^{(4)}} \frac{\max_{\x^{(i_1)}\in \calX,\y^{(i_2)}\in \calY}\lp(\y^{(i_2)})^T
 G\x^{(i_1)}+\|\x^{(i_1)}\|_2\|\y^{(i_2)}\|_2u^{(4)}\rp}{\sqrt{n}}  =  \lim_{\beta\rightarrow\infty} \xi(\calX,\beta,s,1)=
 \lim_{\beta\rightarrow\infty} \psi(\calX,\calY,\beta,s,1)  \\
 \leq   \lim_{\beta\rightarrow\infty} \psi(\calX,\calY,\beta,s,0)=
\mE_{\u^{(2)},\h} \frac{\max_{\x^{(i_1)}\in \calX,\y^{(i_2)}\in \calY} \lp \|\x^{(i_1)}\|_2(\y^{(i_2)})^T\u^{(2)} +\|\y^{(i_2)}\|_2\h^T\x^{(i_1)}\rp}{\sqrt{n}}.
\end{multline}
Connecting beginning and end in (\ref{eq:betainfsplus2}) we obtain a well-known form of the Slepian comparison principle (see, e.g. \cite{Gordon85,Stojnicgscomp16,Slep62}). As stated above, this form is a stronger counterpart of the corresponding result in \cite{Stojnicgscomp16}, and of course only a special case of a much stronger general concept introduced in Theorem \ref{thm:thm1}.

\textbf{\underline{\emph{Numerical results}}}

As in \cite{Stojnicgscomp16}, we below provide a set of numerical results designed to shed a bit more light on $\beta\rightarrow\infty$ regime. The obtained simulation results are shown in Figure \ref{fig:genbetainfsplus1xnorm1psi} and Table \ref{tab:genbetainfsplus1xnorm1psi}. We kept all parameters the same as above ($s=1$ is chosen for the concreteness; such a choice is also in alignment with the choice made in the simulations shown earlier), with only one change. Now, instead of having $\beta=3$ we have $\beta=10$, which in a way emulates $\beta\rightarrow\infty$.
Both, Figure \ref{fig:genbetainfsplus1xnorm1psi} and Table \ref{tab:genbetainfsplus1xnorm1psi}, show an excellent agreement between all presented results. We also note that a fairly small value of $\beta$, namely, $\beta=10$, seems as a pretty solid approximation of $\beta\rightarrow\infty$. This is especially clear from the right part of Figure \ref{fig:genbetainfsplus1xnorm1psi} where one can observe that for $\beta=10$ the resulting curves are much closer to the purple circles (which effectively represent the $\beta\rightarrow\infty$ regime).


\begin{figure}[htb]
\begin{minipage}[b]{.5\linewidth}
\centering
\centerline{\epsfig{figure=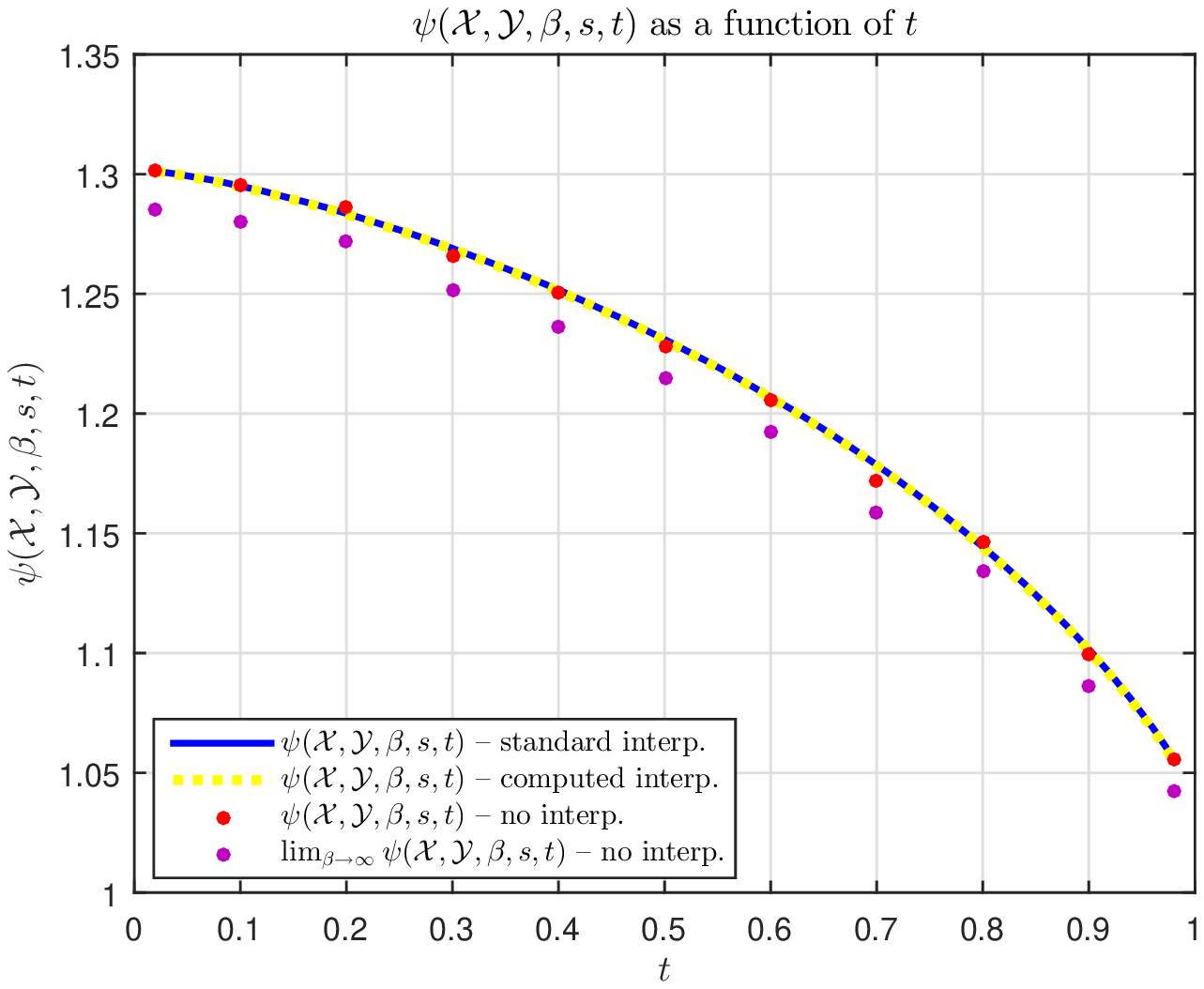,width=9cm,height=7cm}}
\end{minipage}
\begin{minipage}[b]{.5\linewidth}
\centering
\centerline{\epsfig{figure=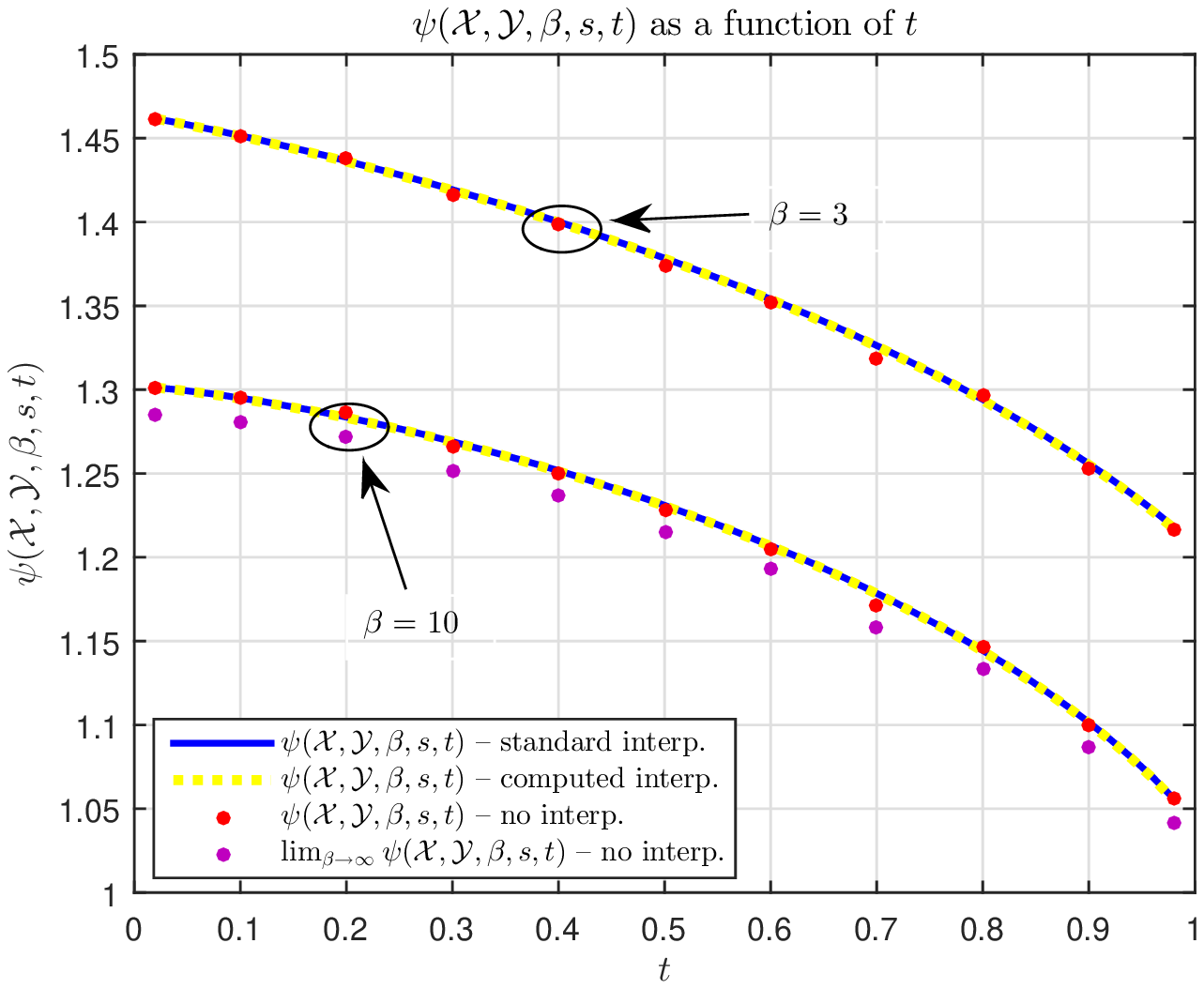,width=9cm,height=7cm}}
\end{minipage}
\caption{Left -- $\psi(\calX,\calY,\beta,s,t)$ as a function of $t$; $m=5$, $n=5$, $l=10$, $\calX=\calX^{+}$, $\calY=\calY^{+}$, $\beta=10$, $s=1$; right -- comparison between $\beta=3$ and $\beta=10$}
\label{fig:genbetainfsplus1xnorm1psi}
\end{figure}

\begin{table}[h]
\caption{Simulated results --- $m=5$, $n=5$, $l=10$, $\calX=\calX^{+}$, $\calY=\calY^{+}$, $\beta=10$, $s=1$}\vspace{.1in}
\hspace{-0in}\centering
\begin{tabular}{||c||c|c|c|c|c|c||}\hline\hline
$ t$  &  $\frac{d\psi}{dt}$; (\ref{eq:genanal11}) & $\frac{d\psi}{dt}$;  (\ref{eq:conalt2}) & $\psi$;  (\ref{eq:genanal11}) and (\ref{eq:co1eq1}) & $\psi$; (\ref{eq:conalt2}) and (\ref{eq:co1eq1}) & $\psi$;  (\ref{eq:genanal8})& $\lim_{\beta\rightarrow\infty}\psi$;  (\ref{eq:genanal8})\\  \hline\hline
$ 0.1000 $ & $ -0.1002 $ & $ -0.0943 $ & $\bl{\mathbf{ 1.2950 }}$ & $\bl{\mathbf{ 1.2946 }}$ & $\mathbf{ 1.2955 }$& $\prp{\mathbf{ 1.2803 }}$\\ \hline
$ 0.2000 $ & $ -0.1287 $ & $ -0.1264 $ & $\bl{\mathbf{ 1.2836 }}$ & $\bl{\mathbf{ 1.2832 }}$ & $\mathbf{ 1.2859 }$& $\prp{\mathbf{ 1.2715 }}$\\ \hline
$ 0.3000 $ & $ -0.1593 $ & $ -0.1580 $ & $\bl{\mathbf{ 1.2690 }}$ & $\bl{\mathbf{ 1.2687 }}$ & $\mathbf{ 1.2659 }$& $\prp{\mathbf{ 1.2519 }}$\\ \hline
$ 0.4000 $ & $ -0.1847 $ & $ -0.1902 $ & $\bl{\mathbf{ 1.2517 }}$ & $\bl{\mathbf{ 1.2512 }}$ & $\mathbf{ 1.2503 }$& $\prp{\mathbf{ 1.2367 }}$\\ \hline
$ 0.5000 $ & $ -0.2181 $ & $ -0.2182 $ & $\bl{\mathbf{ 1.2310 }}$ & $\bl{\mathbf{ 1.2306 }}$ & $\mathbf{ 1.2276 }$& $\prp{\mathbf{ 1.2143 }}$\\ \hline
$ 0.6000 $ & $ -0.2590 $ & $ -0.2511 $ & $\bl{\mathbf{ 1.2069 }}$ & $\bl{\mathbf{ 1.2067 }}$ & $\mathbf{ 1.2054 }$& $\prp{\mathbf{ 1.1924 }}$\\ \hline
$ 0.7000 $ & $ -0.3061 $ & $ -0.3036 $ & $\bl{\mathbf{ 1.1785 }}$ & $\bl{\mathbf{ 1.1783 }}$ & $\mathbf{ 1.1719 }$& $\prp{\mathbf{ 1.1589 }}$\\ \hline
$ 0.8000 $ & $ -0.3628 $ & $ -0.3778 $ & $\bl{\mathbf{ 1.1444 }}$ & $\bl{\mathbf{ 1.1440 }}$ & $\mathbf{ 1.1469 }$& $\prp{\mathbf{ 1.1338 }}$\\ \hline
$ 0.9000 $ & $ -0.4719 $ & $ -0.4776 $ & $\bl{\mathbf{ 1.1016 }}$ & $\bl{\mathbf{ 1.1011 }}$ & $\mathbf{ 1.0997 }$& $\prp{\mathbf{ 1.0865 }}$
\\ \hline\hline
\end{tabular}
\label{tab:genbetainfsplus1xnorm1psi}
\end{table}

\subsubsection{$s<0$ -- reestablishing a Gordon's minmax comparison}
\label{sec:betainfsminus1}

For $s<0$, (\ref{eq:betainf1}) gives
\begin{eqnarray}\label{eq:betainfsminus1}
\lim_{\beta\rightarrow\infty} \xi(\calX,\beta,s,1) & = & \mE_{G,u^{(4)}} \frac{\max_{\x^{(i_1)}\in \calX}\lp-\max_{\y^{(i_2)}\in \calY}\lp(\y^{(i_2)})^T
 G\x^{(i_1)}+\|\x^{(i_1)}\|_2\|\y^{(i_2)}\|_2u^{(4)}\rp\rp}{\sqrt{n}}\nonumber \\
& = & - \mE_{G,u^{(4)}} \frac{\min_{\x^{(i_1)}\in \calX}\max_{\y^{(i_2)}\in \calY}\lp(\y^{(i_2)})^T
 G\x^{(i_1)}+\|\x^{(i_1)}\|_2\|\y^{(i_2)}\|_2u^{(4)}\rp}{\sqrt{n}}.
\end{eqnarray}
We can now again rely on $\xi(\calX,\calY,\beta,s)=\psi(\calX,\calY,\beta,s,1)$ and the above machinery to obtain
\begin{multline}\label{eq:betainfsminus2}
- \mE_{G,u^{(4)}} \frac{\min_{\x^{(i_1)}\in \calX}\max_{\y^{(i_2)}\in \calY}\lp(\y^{(i_2)})^T
 G\x^{(i_1)}+\|\x^{(i_1)}\|_2\|\y^{(i_2)}\|_2u^{(4)}\rp}{\sqrt{n}}  =  \lim_{\beta\rightarrow\infty} \xi(\calX,\calY,\beta,s)\\=
 \lim_{\beta\rightarrow\infty} \psi(\calX,\calY,\beta,s,1)
 \leq   \lim_{\beta\rightarrow\infty} \psi(\calX,\beta,s,0)\\=
\mE_{\u^{(2)},\h} \frac{\max_{\x^{(i_1)}\in \calX} \lp -\max_{\y^{(i_2)}\in \calY}\lp \|\x^{(i_1)}\|_2(\y^{(i_2)})^T\u^{(2)} +\|\y^{(i_2)}\|_2\h^T\x^{(i_1)}\rp\rp}{\sqrt{n}}  \\
 =
- \mE_{\u^{(2)},\h} \frac{\min_{\x^{(i_1)}\in \calX} \lp \max_{\y^{(i_2)}\in \calY}\lp \|\x^{(i_1)}\|_2(\y^{(i_2)})^T\u^{(2)} +\|\y^{(i_2)}\|_2\h^T\x^{(i_1)}\rp\rp}{\sqrt{n}} .
\end{multline}
Connecting beginning and end in (\ref{eq:betainfsminus2}) one obtains a form of the well-known Gordon comparison principle \cite{Gordon85} which is an upgrade on the above mentioned Slepian's comparison principle. As was the case above when we discussed specialization to the Slepian's max principle, (\ref{eq:betainfsminus2}) is a stronger counterpart of the corresponding result in \cite{Stojnicgscomp16}, and only a special case of a much stronger concept presented in Theorem \ref{thm:thm1}.


\textbf{\underline{\emph{Numerical results}}}

In Figure \ref{fig:genbetainfsmin1xnorm1psi} and Table \ref{tab:genbetainfsmin1xnorm1psi} results obtained through  simulations are shown. All parameters are again the same as earlier (this time though, for the concreteness we set $s=-1$). From both, Figure \ref{fig:genbetainfsmin1xnorm1psi} and Table \ref{tab:genbetainfsmin1xnorm1psi}, one can again observe a solid agreement between all the presented results with $\beta=10$ being a pretty good approximation of $\beta\rightarrow\infty$.


\begin{figure}[htb]
\begin{minipage}[b]{.5\linewidth}
\centering
\centerline{\epsfig{figure=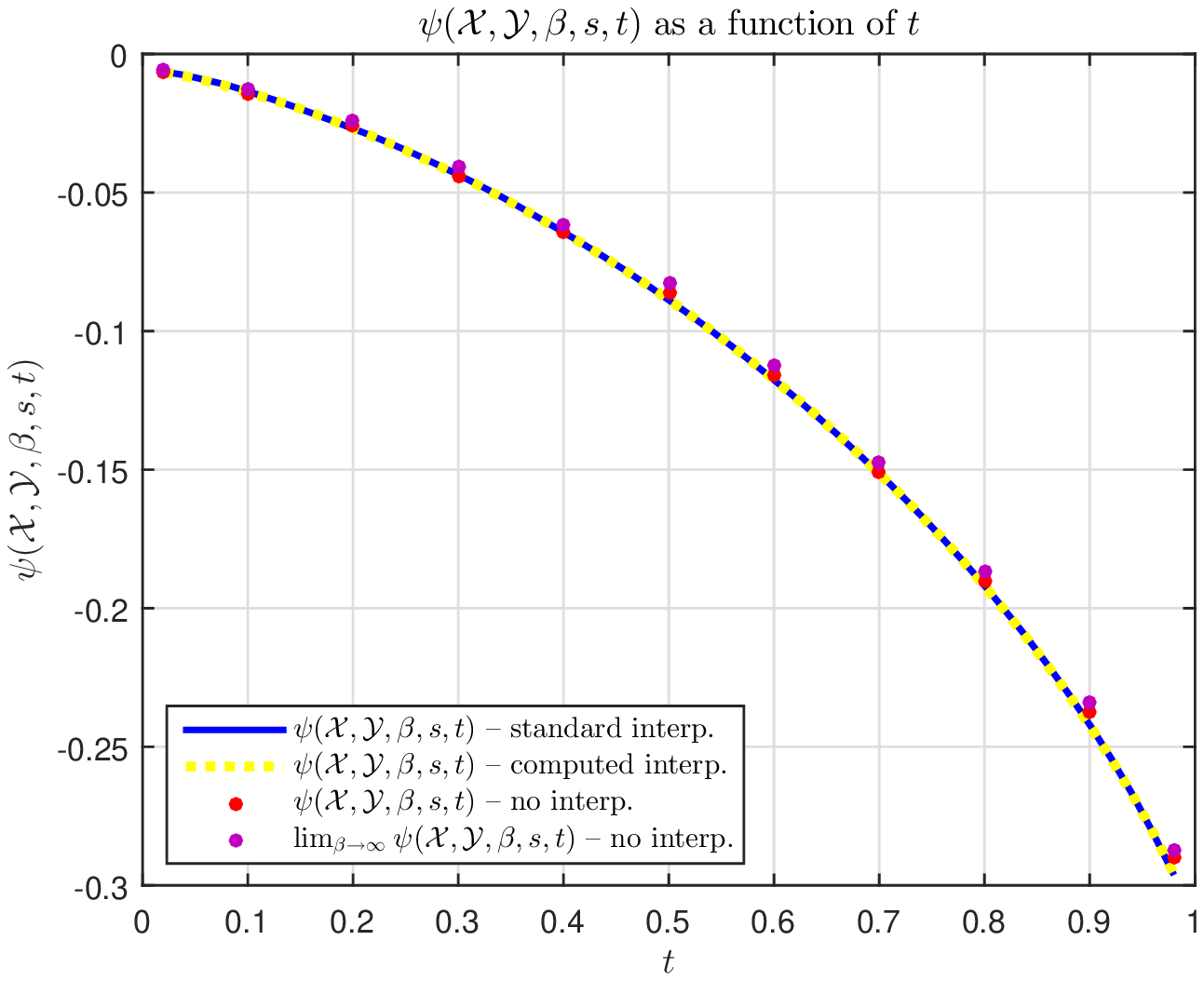,width=9cm,height=7cm}}
\end{minipage}
\begin{minipage}[b]{.5\linewidth}
\centering
\centerline{\epsfig{figure=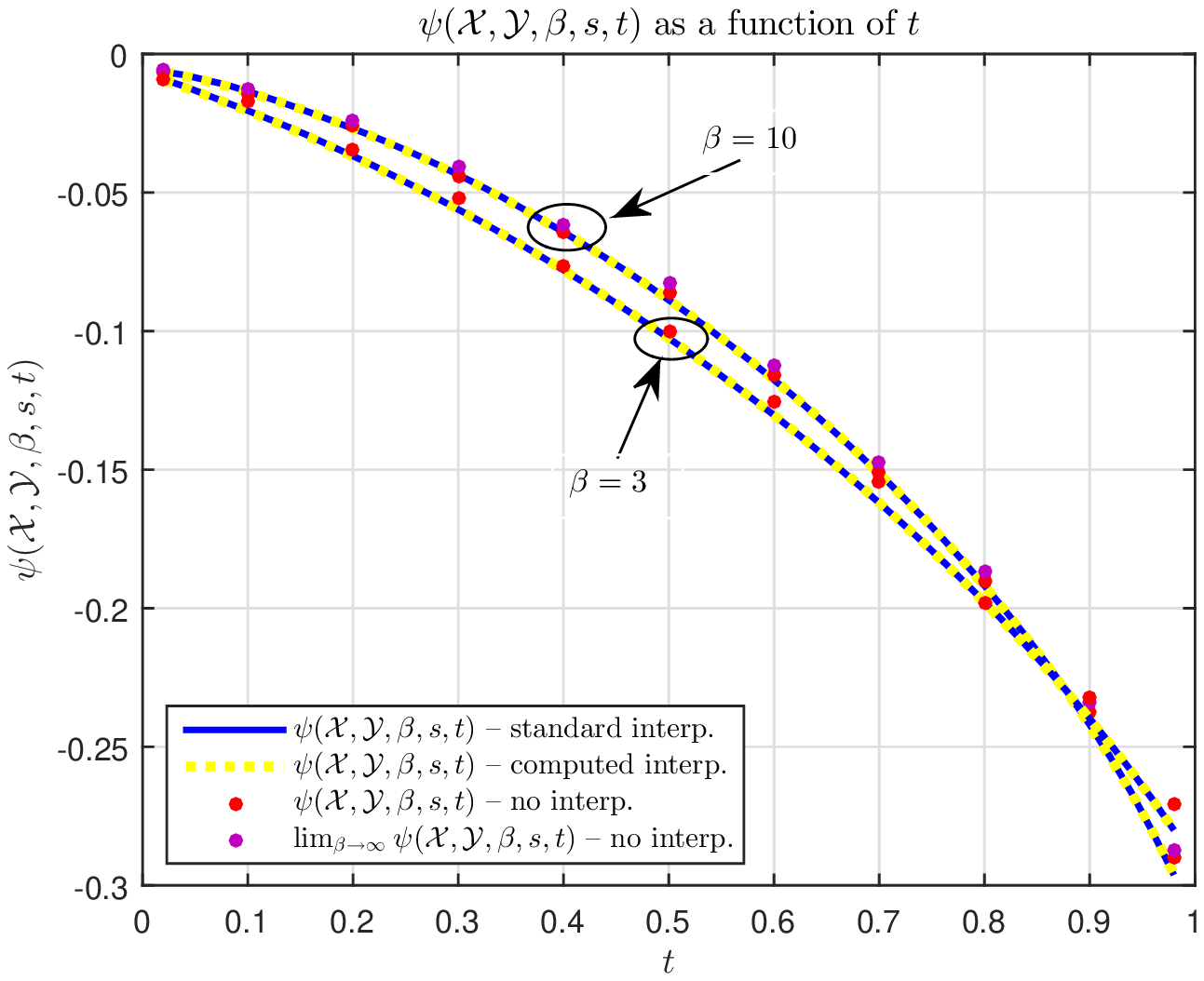,width=9cm,height=7cm}}
\end{minipage}
\caption{Left -- $\psi(\calX,\calY,\beta,s,t)$ as a function of $t$; $m=5$, $n=5$, $l=10$, $\calX=\calX^{+}$, $\calY=\calY^{+}$, $\beta=10$, $s=-1$; right -- comparison between $\beta=3$ and $\beta=10$}
\label{fig:genbetainfsmin1xnorm1psi}
\end{figure}

\begin{table}[h]
\caption{Simulated results --- $m=5$, $n=5$, $l=10$, $\calX=\calX^{+}$, $\beta=10$, $s=-1$}\vspace{.1in}
\hspace{-0in}\centering
\begin{tabular}{||c||c|c|c|c|c|c||}\hline\hline
$ t$  &  $\frac{d\psi}{dt}$; (\ref{eq:genanal11}) & $\frac{d\psi}{dt}$;  (\ref{eq:conalt2}) & $\psi$;  (\ref{eq:genanal11}) and (\ref{eq:co1eq1}) & $\psi$; (\ref{eq:conalt2}) and (\ref{eq:co1eq1}) & $\psi$;  (\ref{eq:genanal8})& $\lim_{\beta\rightarrow\infty}\psi$;  (\ref{eq:genanal8})\\  \hline\hline
$ 0.1000 $ & $ -0.1196 $ & $ -0.1069 $ & $\bl{\mathbf{ -0.0135 }}$ & $\bl{\mathbf{ -0.0138 }}$ & $\mathbf{ -0.0141 }$& $\prp{\mathbf{ -0.0128 }}$\\ \hline
$ 0.2000 $ & $ -0.1437 $ & $ -0.1478 $ & $\bl{\mathbf{ -0.0270 }}$ & $\bl{\mathbf{ -0.0269 }}$ & $\mathbf{ -0.0261 }$& $\prp{\mathbf{ -0.0241 }}$\\ \hline
$ 0.3000 $ & $ -0.1839 $ & $ -0.1840 $ & $\bl{\mathbf{ -0.0436 }}$ & $\bl{\mathbf{ -0.0438 }}$ & $\mathbf{ -0.0438 }$& $\prp{\mathbf{ -0.0411 }}$\\ \hline
$ 0.4000 $ & $ -0.2203 $ & $ -0.2194 $ & $\bl{\mathbf{ -0.0643 }}$ & $\bl{\mathbf{ -0.0644 }}$ & $\mathbf{ -0.0646 }$& $\prp{\mathbf{ -0.0615 }}$\\ \hline
$ 0.5000 $ & $ -0.2642 $ & $ -0.2587 $ & $\bl{\mathbf{ -0.0888 }}$ & $\bl{\mathbf{ -0.0887 }}$ & $\mathbf{ -0.0861 }$& $\prp{\mathbf{ -0.0826 }}$\\ \hline
$ 0.6000 $ & $ -0.3036 $ & $ -0.3070 $ & $\bl{\mathbf{ -0.1176 }}$ & $\bl{\mathbf{ -0.1175 }}$ & $\mathbf{ -0.1160 }$& $\prp{\mathbf{ -0.1120 }}$\\ \hline
$ 0.7000 $ & $ -0.3563 $ & $ -0.3662 $ & $\bl{\mathbf{ -0.1514 }}$ & $\bl{\mathbf{ -0.1514 }}$ & $\mathbf{ -0.1511 }$& $\prp{\mathbf{ -0.1471 }}$\\ \hline
$ 0.8000 $ & $ -0.4352 $ & $ -0.4442 $ & $\bl{\mathbf{ -0.1918 }}$ & $\bl{\mathbf{ -0.1920 }}$ & $\mathbf{ -0.1905 }$& $\prp{\mathbf{ -0.1867 }}$\\ \hline
$ 0.9000 $ & $ -0.5560 $ & $ -0.5548 $ & $\bl{\mathbf{ -0.2420 }}$ & $\bl{\mathbf{ -0.2424 }}$ & $\mathbf{ -0.2378 }$& $\prp{\mathbf{ -0.2343 }}$
\\ \hline\hline
\end{tabular}
\label{tab:genbetainfsmin1xnorm1psi}
\end{table}

\section{A lifting procedure}
\label{sec:lifting}

We start again with sets $\calX$ and $\calY$ and consider the following function
\begin{eqnarray}\label{eq:liftgenanal1}
 f_*(G,u^{(4)},\calX,\calY,\beta,s)= \lp \sum_{i_1=1}^{l}\lp\sum_{i_2=1}^{l}e^{\beta \lp (\y^{(i_2)})^T
 G\x^{(i_1)}+\|\x^{(i_1)}\|_2\|\y^{(i_2)}\|_2 u^{(4)}\rp} \rp^{s}\rp^{c_3},
\end{eqnarray}
where all quantities are as in Section \ref{sec:gencon} and $c_3>0$ is a real parameter. Following (\ref{eq:liftgenanal2}) we then introduce
\begin{eqnarray}\label{eq:liftgenanal2}
\xi_*(\calX,\calY,\beta,s)  \triangleq   \mE_{G,u^{(4)}} f(G,u^{(4)},\calX,\calY,\beta,s),
\end{eqnarray}
and consider the following interpolating function $\psi_*(\cdot)$ as an object convenient for studying properties of $\xi_*(\calX,\calY,\beta,s)$
\begin{multline}\label{eq:liftgenanal3}
\psi_*(\calX,\calY,\beta,s,t)  =  \mE_{G,u^{(4)},\u^{(2)},\h}  \\
 \times \lp \sum_{i_1=1}^{l}\lp\sum_{i_2=1}^{l}e^{\beta \lp \sqrt{t}(\y^{(i_2)})^T
 G\x^{(i_1)}+\sqrt{1-t}\|\x^{(i_2)}\|_2 (\y^{(i_2)})^T\u^{(2)}+\sqrt{t}\|\x^{(i_1)}\|_2\|\y^{(i_2)}\|_2,u^{(4)} +\sqrt{1-t}\|\y^{(i_2)}\|_2\h^T\x^{(i)}\rp} \rp^{s}\rp^{c_3},
\end{multline}
where again, all quantities are exactly the same as earlier with the above mentioned addition of $c_3$. Following  (\ref{eq:liftgenanal8}) (and clearly relying on (\ref{eq:genanal7})) we write
\begin{eqnarray}\label{eq:liftgenanal8}
\psi_*(\calX,\calY,\beta,s,t) & = &  \mE_{\u^{(i_1,1)},\u^{(2)},\u^{(i_1,3)},u^{(4)}} Z^{c_3}.
\end{eqnarray}
Following further the strategy of Section \ref{sec:gencon}, below we study the monotonicity of $\psi_*(\calX,\calY,\beta,s,t)$ when viewed as a function of $t$. As it will be soon clear, many of the results that we created in Section \ref{sec:gencon} with fairly straightforward modifications will be applicable here as well. As usual, we will try to skip all the details that remain the same and instead will put an emphasis on those that bring a difference. We start with the following derivative (basically an analogous version of (\ref{eq:genanal9}))
\begin{eqnarray}\label{eq:liftgenanal9}
\frac{d\psi_*(\calX,\calY,\beta,s,t)}{dt} & = &  \mE_{\u^{(i_1,1)},\u^{(2)},\u^{(i_1,3)},u^{(4)}} \frac{dZ^{c_3}}{dt}\nonumber \\
& = &  \mE_{\u^{(i_1,1)},\u^{(2)},\u^{(i_1,3)},u^{(4)}} \frac{sc_3}{Z^{1-c_3}}  \sum_{i=1}^{l} (C^{(i_1)})^{s-1} \nonumber \\
& & \times \sum_{i_2=1}^{l}\beta_{i_1}A^{(i_1,i_2)}\lp \frac{dB^{(i_1,i_2)}}{dt}+\frac{\|\y^{(i_2)}\|_2 u^{(4)}}{2\sqrt{t}}-\frac{\|\y^{(i_2)}\|_2 \u^{(i_1,3)}}{2\sqrt{1-t}}\rp.
\end{eqnarray}
Relying on (\ref{eq:genanal10}) we further find
\begin{eqnarray}\label{eq:liftgenanal11}
\frac{d\psi_*(\calX,\calY,\beta,s,t)}{dt}
& = &  \mE_{\u^{(i_1,1)},\u^{(2)},\u^{(i_1,3)},u^{(4)}} \frac{sc_3}{Z^{1-c_3}}  \sum_{i_1=1}^{l} (C^{(i_1)})^{s-1} \nonumber \\
& & \times \sum_{i_2=1}^{l}\beta_{i_1}A^{(i_1,i_2)}\lp \sum_{j=1}^{m}\lp \frac{\y_j^{(i_2)}\u_j^{(i_1,1)}}{2\sqrt{t}}-\frac{\y_j^{(i_2)}\u_j^{(2)}}{2\sqrt{1-t}}\rp+\frac{\|\y^{(i_2)}\|_2 u^{(4)}}{2\sqrt{t}}-\frac{\|\y^{(i_2)}\|_2 \u^{(i_1,3)}}{2\sqrt{1-t}}\rp.\nonumber \\
\end{eqnarray}
As in Section \ref{sec:gencon}, each of the terms in the above sum can be handled separately. However, this time the calculations will be done in a much faster fashion as one can utilize quite a few of the results already obtained earlier.

\subsection{Computing $\frac{d\psi_*(\calX,\calY,\beta,s,t)}{dt}$}
\label{sec::liftder}

As mentioned above, we will split the computation into several parts.

\underline{\textbf{\emph{1) Determining $\mE_{\u^{(i_1,1)},\u^{(2)},\u^{(i_1,3)},u^{(4)}}  \frac{(C^{(i_1)})^{s-1} A^{(i_1,i_2)}\u_j^{(i_1,1)}\y_j^{(i_2)}}{Z^{1-c_3}}$
}}}

Now, the key observation that we will employ here (and quite a few times again below) is that all the main calculations from Section \ref{sec:gencon} can be repeated and not only conceptually but pretty much literally with very small modifications. These modifications will be in the powers of $Z$ and the constants that multiply them. Namely where we used to have $Z$ in Section \ref{sec:hand1} we will now have $Z^{1-c_3}$ and where we used to have $-Z^{-2}$ we will now have $(c_3-1)Z^{c_3-2}$. All other adjustments are trivial and one finds
\begin{multline}\label{eq:liftgenanal19}
\mE_{\u^{(i_1,1)},\u^{(2)},\u^{(i_1,3)},u^{(4)}}  \frac{(C^{(i_1)})^{s-1} A^{(i_1,i_2)}\u_j^{(i_1,1)}\y_j^{(i_2)}}{Z^{1-c_3}}  \\
 =
\mE \lp \frac{\y_j^{(i_2)}}{Z^{1-c_3}}\lp(C^{(i_1)})^{s-1}\beta_{i_1}A^{(i_1,i_2)}\y_j^{(i_2)}\sqrt{t}+(s-1)(C^{(i_1)})^{s-2}\beta_{i_1}\sum_{p_2=1}^{l}A^{(i_1,p_2)}\y_j^{(p_2)}\sqrt{t}\rp \rp \\
-(1-c_3)
\mE \lp\sum_{p_1=1}^{l} \frac{(\x^{(i_1)})^T\x^{(p_1)}}{\|\x^{(i_1)}\|_2\|\x^{(p_1)}\|_2}
\frac{(C^{(i_1)})^{s-1} A^{(i_1,i_2)}\y_j^{(i_2)}}{Z^{2-c_3}}
s  (C^{(p_1)})^{s-1}\sum_{p_2=1}^{l}\beta_{p_1}A^{(p_1,p_2)}\y_j^{(p_2)}\sqrt{t}\rp.
\end{multline}

\underline{\textbf{\emph{2) Determining $\mE_{\u^{(i_1,1)},\u^{(2)},\u^{(i_1,3)},u^{(4)}}  \frac{(C^{(i_1)})^{s-1} A^{(i_1,i_2)}\u_j^{(2)}\y_j^{(i_2)}}{Z^{1-c_3}}$
}}}

Repeating all the calculations from Section \ref{sec:hand2} with the above mentioned modifications we also find
\begin{multline}\label{eq:liftgenAanal19}
\mE_{\u^{(i_1,1)},\u^{(2)},\u^{(i_1,3)},u^{(4)}}  \frac{(C^{(i_1)})^{s-1} A^{(i_1,i_2)}\u_j^{(2)}\y_j^{(i_2)}}{Z^{1-c_3}}  \\
 =
\mE \lp\frac{\y_j^{(i_2)}}{Z^{1-c_3}}\lp(C^{(i_1)})^{s-1}\beta_{i_1}A^{(i_1,i_2)}\y_j^{(i_2)}\sqrt{1-t}+(s-1)(C^{(i_1)})^{s-2}\beta_{i_1}\sum_{p_2=1}^{l}A^{(i_1,p_2)}\y_j^{(p_2)}\sqrt{1-t}\rp \rp \\
-(1-c_3)
\mE \lp\sum_{p_1=1}^{l}
\frac{(C^{(i_1)})^{s-1} A^{(i_1,i_2)}\y_j^{(i_2)}}{Z^{2-c_3}}
s  (C^{(p_1)})^{s-1}\sum_{p_2=1}^{l}\beta_{p_1}A^{(p_1,p_2)}\y_j^{(p_2)}\sqrt{1-t}\rp.
\end{multline}

\underline{\textbf{\emph{3) Determining $\mE_{\u^{(i_1,1)},\u^{(2)},\u^{(i_1,3)},u^{(4)}}  \frac{(C^{(i_1)})^{s-1} A^{(i_1,i_2)}\u^{(i_1,3)}}{Z^{1-c_3}}$
}}}

Similarly to what we did above, one can also repeat all the calculations from Section \ref{sec:hand3} while accounting for the above mentioned change of powers and multiplying constants we have
the following analogue of (\ref{eq:genBanal19})
\begin{multline}\label{eq:liftgenBanal19}
\mE_{\u^{(i_1,1)},\u^{(2)},\u^{(i_1,3)},u^{(4)}}  \frac{(C^{(i_1)})^{s-1} A^{(i_1,i_2)}\u^{(i_1,3)}}{Z^{1-c_3}}  \\
 =
\mE \lp\frac{1}{Z^{1-c_3}}\lp(C^{(i_1)})^{s-1}\beta_{i_1}A^{(i_1,i_2)}\|\y^{(i_2)}\|_2\sqrt{1-t}+(s-1)(C^{(i_1)})^{s-2}\beta_{i_1}\sum_{p_2=1}^{l}A^{(i_1,p_2)}\|\y^{(p_2)}\|_2\sqrt{1-t}\rp \rp \\
-{1-c_3}
\mE \lp\sum_{p_1=1}^{l}\frac{(\x^{(i_1)})^T\x^{(p_1)}}{\|\x^{(i_1)}\|_2\|\x^{(p_1)}\|_2}
\frac{(C^{(i_1)})^{s-1} A^{(i_1,i_2)}}{Z^{2-c_3}}
s  (C^{(p_1)})^{s-1}\sum_{p_2=1}^{l}\beta_{p_1}A^{(p_1,p_2)}\|\y^{(p_2)}\|_2\sqrt{1-t}\rp.
\end{multline}

\underline{\textbf{\emph{4) Determining $\mE_{\u^{(i_1,1)},\u^{(2)},\u^{(i_1,3)},u^{(4)}}  \frac{(C^{(i_1)})^{s-1} A^{(i_1,i_2)}u^{(4)}}{Z^{1-c_3}}$
}}}

Finally, after repeating all the calculations from Section \ref{sec:hand4} we have the following analogue to (\ref{eq:genCanal19})
\begin{multline}\label{eq:liftgenCanal19}
\mE_{\u^{(i_1,1)},\u^{(2)},\u^{(i_1,3)},u^{(4)}}  \frac{(C^{(i_1)})^{s-1} A^{(i_1,i_2)}u^{(4)}}{Z^{1-c_3}}  \\
 =
\mE \lp\frac{1}{Z^{1-c_3}}\lp(C^{(i_1)})^{s-1}\beta_{i_1}A^{(i_1,i_2)}\|\y^{(i_2)}\|_2\sqrt{t}+(s-1)(C^{(i_1)})^{s-2}\beta_{i_1}\sum_{p_2=1}^{l}A^{(i_1,p_2)}\|\y^{(p_2)}\|_2\sqrt{t}\rp \rp \\
-{1-c_3}
\mE \lp
\frac{(C^{(i_1)})^{s-1} A^{(i_1,i_2)}}{Z^{2-c_3}}
s  \sum_{p_1=1}^{l} (C^{(p_1)})^{s-1}\sum_{p_2=1}^{l}\beta_{p_1}A^{(p_1,p_2)}\|\y^{(p_2)}\|_2\sqrt{t}\rp.
\end{multline}

\underline{\textbf{\emph{Connecting everything together
}}}

Combining (\ref{eq:liftgenanal11}), (\ref{eq:liftgenanal19}), (\ref{eq:liftgenAanal19}), (\ref{eq:liftgenBanal19}), and (\ref{eq:liftgenCanal19}) we can also establish the following set of results (basically fairly similar to the corresponding set obtained in Section \ref{sec:conalt})
\begin{eqnarray}\label{eq:liftconalt1}
\frac{d\psi_*(\calX,\calY,\beta,s,t)}{dt}
 =  \frac{sc_3(1-c_3)}{2} \mE_{\u^{(i_1,1)},\u^{(2)},\u^{(i_1,3)},u^{(4)}} (-S_1+S_2+S_3-S_4)
\end{eqnarray}
where
\begin{eqnarray}
\label{eq:liftconalt1a}
S_{1,*} & = & \sum_{i_1=1}^{l} \sum_{i_2=1}^{l}\beta_{i_1}\sum_{j=1}^{m}
\lp\sum_{p_1=1}^{l} \frac{(\x^{(i_1)})^T\x^{(p_1)}}{\|\x^{(i_1)}\|_2\|\x^{(p_1)}\|_2}
\frac{(C^{(i_1)})^{s-1} A^{(i_1,i_2)}\y_j^{(i_2)}}{Z^{2-c_3}}
s  (C^{(p_1)})^{s-1}\sum_{p_2=1}^{l}\beta_{p_1}A^{(p_1,p_2)}\y_j^{(p_2)}\rp\nonumber \\
S_{2,*} & = & \sum_{i_1=1}^{l} \sum_{i_2=1}^{l}\beta_{i_1}\sum_{j=1}^{m}
 \lp\sum_{p_1=1}^{l}
\frac{(C^{(i_1)})^{s-1} A^{(i_1,i_2)}\y_j^{(i_2)}}{Z^{2-c_3}}
s  (C^{(p_1)})^{s-1}\sum_{p_2=1}^{l}\beta_{p_1}A^{(p_1,p_2)}\y_j^{(p_2)}\rp\nonumber \\
S_{3,*} & = & \sum_{i_1=1}^{l} \sum_{i_2=1}^{l}\beta_{i_1}
\|\y^{(i_2)}\|_2 \lp\sum_{p_1=1}^{l}\frac{(\x^{(i_1)})^T\x^{(p_1)}}{\|\x^{(i_1)}\|_2\|\x^{(p_1)}\|_2}
\frac{(C^{(i_1)})^{s-1} A^{(i_1,i_2)}}{Z^{2-c_3}}
s  (C^{(p_1)})^{s-1}\sum_{p_2=1}^{l}\beta_{p_1}A^{(p_1,p_2)}\|\y^{(p_2)}\|_2\rp\nonumber \\
S_{4,*} & = & \sum_{i_1=1}^{l} \sum_{i_2=1}^{l}\beta_{i_1}
\|\y^{(i_2)}\|_2 \lp
\frac{(C^{(i_1)})^{s-1} A^{(i_1,i_2)}}{Z^{2-c_3}}
s \sum_{p_1=1}^{l}  (C^{(p_1)})^{s-1}\sum_{p_2=1}^{l}\beta_{p_1}A^{(p_1,p_2)}\|\y^{(p_2)}\|_2\rp.
\end{eqnarray}
Repeating (\ref{eq:conalt1b}) and (\ref{eq:conalt1c})
and combining these steps with (\ref{eq:liftconalt1a}) we finally obtain
\begin{eqnarray}\label{eq:liftconalt2}
\frac{d\psi_*(\calX,\calY,\beta,s,t)}{dt}
 & = &  -\frac{s^2\beta^2c_3(1-c_3)}{2} \nonumber \\
 & & \times \mE_{\u^{(i_1,1)},\u^{(2)},\u^{(i_1,3)},u^{(4)}}  \sum_{i_1=1}^{l} \sum_{p_1=1}^{l}
\frac{(C^{(i_1)})^{s}(C^{(p_1)})^{s}(\|\x^{(i_1)}\|_2\|\x^{(p_1)}\|_2-(\x^{(i_1)})^T\x^{(p_1)})}{Z^{2-c_3}} \nonumber \\
& & \times \lp\sum_{i_2=1}^{l}\sum_{p_2=1}^{l}
\frac{A^{(i_1,i_2)}A^{(p_1,p_2)}}{C^{(i_1)}C^{(p_1)}}
(\|\y^{(i_2)}\|_2\|\y^{(p_2)}\|_2-(\y^{(i_2)})^T\y^{(p_2)})\rp.
\end{eqnarray}
Depending on the value of $c_3$ one can now discuss the sign of $\frac{d\psi_*(\calX,\calY,\beta,s,t)}{dt}$ and whether function $\psi_*(\calX,\calY,\beta,s,t)$ is non-increasing (decreasing) or non-decreasing (increasing) in $t$. The obtained results are summarized in the following theorem and its a corollary.
\begin{theorem}
\label{thm:liftthm2}
  Assume the setup of Theorem \ref{thm:thm1}. We then have
\begin{eqnarray}\label{eq:liftco1eq1}
\psi_*(\calX,\calY,\beta,s,c_3,t)= \psi_*(\calX,\calY,\beta,s,c_3,0)+\int_{0}^{t}\frac{d\psi_*(\calX,\calY,\beta,s,c_3,t)}{dt}dt,
\end{eqnarray}
where $\frac{d\psi_*(\calX,\calY,\beta,s,c_3,t)}{dt}$ is given by (\ref{eq:liftconalt2}).
\end{theorem}
\begin{proof}
  Follows automatically through the above discussion.
\end{proof}
\begin{corollary}\label{cor:liftcor1}
  Assume the setup of Theorem \ref{thm:liftthm2}.

  1) If $0< c_3< 1$ then $\frac{d\psi_*(\calX,\calY,\beta,s,c_3,t)}{dt}<0$ and $\psi_*(\calX,\calY,\beta,s,c_3,t)$ is decreasing in $t$ and one finds the following comparison principle
\begin{eqnarray}\label{eq:liftco2aeq1}
\lim_{\beta\rightarrow\infty}\psi_*(\calX,\calY,\beta,s,c_3,0) \geq \lim_{\beta\rightarrow\infty} \psi_*(\calX,\calY,\beta,s,c_3,t)\geq \lim_{\beta\rightarrow\infty}\psi_*(\calX,\calY,\beta,s,c_3,1).
\end{eqnarray}

  2) If $c_3> 1$ or $c_3< 0$ then $\frac{d\psi_*(\calX,\calY,\beta,s,c_3,t)}{dt}>0$ and $\psi_*(\calX,\calY,\beta,s,c_3,t)$ is increasing in $t$ and one finds the following comparison principle
\begin{eqnarray}\label{eq:liftco2aeq2}
\lim_{\beta\rightarrow\infty}\psi_*(\calX,\calY,\beta,s,c_3,0) \leq \lim_{\beta\rightarrow\infty} \psi_*(\calX,\calY,\beta,s,c_3,t)\leq \lim_{\beta\rightarrow\infty}\psi_*(\calX,\calY,\beta,s,c_3,1).
\end{eqnarray}
\end{corollary}
\begin{proof}
  Follows again automatically by the arguments presented above.
\end{proof}

\subsection{$\beta\rightarrow \infty$}
\label{sec:liftbetainf}

Following into the footsteps of \cite{Stojnicgscomp16}, in this section we discuss in a bit deeper detail one of the key consequences of the lifting procedure introduced above (as it will be soon clear, it will connect to some of the comparison principles that we utilized in e.g. \cite{StojnicLiftStrSec13,StojnicMoreSophHopBnds10,StojnicRicBnds13}). As in \cite{Stojnicgscomp16}, we will assume that $\beta$ is large, say $\beta\rightarrow\infty$ and that the scaling $c_3\leftarrow \frac{c^{(s)}_3}{\beta}$, where $c^{(s)}_3$ is a finite positive real number, is in place as well. Clearly, one then has $c_3(1-c_3)\geq 0$ which implies $\frac{\psi_*(\calX,\calY,\beta,s,c_3,t)}{dt}\leq 0$. That, on the other hand, also means that function $\psi_*(\calX,\calY,\beta,s,c_3,t)$ is decreasing in $t$. We summarize this adaptation into the following corollary of Theorem \ref{thm:liftthm2}.
\begin{corollary}\label{cor:liftcor2}
  Assume the setup of Theorem \ref{thm:liftthm2}. Let $c_3\leftarrow \frac{c^{(s)}_3}{\beta}$, where $c^{(s)}_3$ is a finite positive real number. Then $\psi_*(\calX,\calY,\beta,s,c_3,t)$ is decreasing in $t$ and we have the following comparison principle
\begin{eqnarray}\label{eq:liftliftco2eq2}
\lim_{\beta\rightarrow\infty}\psi_*(\calX,\calY,\beta,s,\frac{c^{(s)}_3}{\beta},0) \geq \lim_{\beta\rightarrow\infty} \psi_*(\calX,\calY,\beta,s,\frac{c^{(s)}_3}{\beta},t)\geq \lim_{\beta\rightarrow\infty}\psi_*(\calX,\calY,\beta,s,\frac{c^{(s)}_3}{\beta},1).
\end{eqnarray}
\end{corollary}
\begin{proof}
  Follows automatically by the above arguments.
\end{proof}
Paralleling further \cite{Stojnicgscomp16}, we below also study the following limiting behavior of $\xi_*(\calX,\beta,s,\frac{c^{(s)}_3}{\beta})$, i.e.
\begin{eqnarray}\label{eq:liftliftbetainf1}
 \log\lim_{\beta\rightarrow\infty} \xi_*(\calX,\calY,\beta,s,\frac{c^{(s)}_3}{\beta})
& = & \log\lim_{\beta\rightarrow\infty} \mE_{G,u^{(4)}}\lp \sum_{i_1=1}^{l}\lp\sum_{i_2=1}^{l}e^{\beta \lp (\y^{(i_2)})^T
 G\x^{(i_1)}+\|\x^{(i_1)}\|_2\|\y^{(i_2)}\|_2 u^{(4)}\rp} \rp^{s}\rp^{\frac{c_3^{(s)}}{\beta}}\nonumber \\
 & = &  \log \mE_{G,u^{(4)}}\lp e^{c^{(s)}_3\max_{\x^{(i_1)}\in\calX} s\max_{\y^{(i_2)}\in\calY}\lp (\y^{(i_2)})^T
 G\x^{(i_1)}+\|\x^{(i_1)}\|_2\|\y^{(i_2)}\|_2 u^{(4)}\rp} \rp.\nonumber \\
\end{eqnarray}

\subsubsection{$s=1$ -- a lifted Slepian's (fully bilinear) max comparison}
\label{sec:liftliftbetainfsplus1}

Choosing $s=1$ in (\ref{eq:liftliftbetainf1}) gives
\begin{eqnarray}\label{eq:liftliftbetainfsplus1}
 \log\lim_{\beta\rightarrow\infty} \xi_*(\calX,\calY,\beta,1,\frac{c^{(s)}_3}{\beta})
 & = &  \log \mE_{G,u^{(4)}}\lp e^{c^{(s)}_3\max_{\x^{(i_1)}\in\calX} \max_{\y^{(i_2)}\in\calY}\lp (\y^{(i_2)})^T
 G\x^{(i_1)}+\|\x^{(i_1)}\|_2\|\y^{(i_2)}\|_2 u^{(4)}\rp} \rp.\nonumber \\
\end{eqnarray}
Now, we recall that $\xi_*(\calX,\calY,\beta,1,\frac{c^{(s)}_3}{\beta})=\psi_*(\calX,\calY,\beta,1,\frac{c^{(s)}_3}{\beta},1)$ and based on the above we also find
\begin{multline}\label{eq:liftliftbetainfsplus2}
\log \mE_{G,u^{(4)}}\lp e^{c^{(s)}_3\max_{\x^{(i_1)}\in\calX} \max_{\y^{(i_2)}\in\calY}\lp (\y^{(i_2)})^T
 G\x^{(i_1)}+\|\x^{(i_1)}\|_2\|\y^{(i_2)}\|_2 u^{(4)}\rp} \rp  =
 \log\lim_{\beta\rightarrow\infty} \psi_*(\calX,\calY,\beta,1,\frac{c^{(s)}_3}{\beta},1) \\
 \leq     \log\lim_{\beta\rightarrow\infty} \psi_*(\calX,\calY,\beta,1,\frac{c^{(s)}_3}{\beta},t)
 \leq     \log\lim_{\beta\rightarrow\infty} \psi_*(\calX,\calY,\beta,1,\frac{c^{(s)}_3}{\beta},0) \\
 =
\log \mE_{\u^{(2)},\h}\lp e^{c^{(s)}_3\lp \max_{\x^{(i_1)}\in\calX}\max_{\y^{(i_2)}\in\calY}\lp \|\x^{(i_1)}\|_2(\y^{(i_2)})^T\u^{(2)}+\|\y^{(i_2)}\|_2\h^T\x^{(i_1)}\rp\rp} \rp.
\end{multline}
Taking beginning and end in (\ref{eq:liftliftbetainfsplus2}) establishes basically the same comparison that we utilized in \cite{StojnicMoreSophHopBnds10}, which is the following Gordon's upgrade of the Slepian's (so to say fully bilinear) max principle
\begin{multline}\label{eq:liftliftbetainfsplus3}
 \log \mE_{G,u^{(4)}} e^{c^{(s)}_3\max_{\x^{(i_1)}\in \calX,\y^{(i_2)}\in \calY} \lp (\y^{(i_2)})^T G\x^{(i_1)}+\|\x^{(i_1)}\|_2\|\y^{(i_2)}\|_2u^{(4)}\rp} \\
    \leq \log \mE_{\u^{(2)},\h} e^{c^{(s)}_3\max_{\x^{(i_1)}\in \calX,\y^{(i_2)}\in \calY} \lp \|\x^{(i_1)}\|_2(\y^{(i_2)})^T \u^{(2)}+ \|\y^{(i_2)}\|_2\h^T\x^{(i_1)}\rp}.
\end{multline}
Similarly to what we observed in \cite{Stojnicgscomp16}, (\ref{eq:liftliftbetainfsplus1}) and (\ref{eq:liftliftbetainfsplus2}) can be viewed as a lifted Slepian (fully bilinear) max comparison principle. As discussed in \cite{Stojnicgscomp16} (see also, e.g. \cite{StojnicMoreSophHopBnds10}) the above lifting procedure is often the only known tool that can significantly improve on the original Slepian's principle (needles to say, Theorem \ref{thm:liftthm2} is a much stronger concept of which the above form is only a special case).

\textbf{\underline{\emph{Numerical results}}}

We also conducted a set of numerical experiments to complement the theoretical results that we presented above. The numerical results that we obtained through these experiments are shown in Figure \ref{fig:liftedbetainfsplus1xnorm1psi} and Table \ref{tab:liftedbetainfsplus1xnorm1psi}. We selected all parameters as in Section \ref{sec:gencon} with $\beta=10$ as a way to emulate $\beta\rightarrow\infty$ and $c_3=.1$ (to get a bit more reliable results, we this time averaged all random quantities over a set of $8e4$ experiments). The right part of the figure also contains how the obtained results compare to the same scenario but with no lifting. To have that comparison make sense, as in \cite{Stojnicgscomp16}, we worked with the adjusted $\psi_*(\cdot)$. Basically, in Table \ref{tab:liftedbetainfsplus1xnorm1psi}, the values for $\psi_*(\calX,\calY,\beta,s,c_3,t)$ are given in two forms: 1) the value itself and 2) the adjusted value $\lp\frac{1}{\beta |s| c_3}\log\lp \psi_*(\calX,\calY,\beta,s,c_3,t)\rp-\frac{\beta|s| c_3}{2}\rp/\sqrt{n}$ (as in \cite{Stojnicgscomp16}, the adjusted value acts in a way as a connection between $\psi_*(\calX,\calY,\beta,s,c_3,t)$ and $\psi(\calX,\calY,\beta,s,t)$). As can be seen from both, Figure \ref{fig:liftedbetainfsplus1xnorm1psi} and Table \ref{tab:liftedbetainfsplus1xnorm1psi}, there is a solid agreement between all the presented results. Moreover, $\beta=10$ seems to be a solid approximation of $\beta\rightarrow\infty$ (the values for $\lim_{\beta\rightarrow\infty}\psi_*$ were obtained with $c_3^{(s)}=c_3\beta$ so that the fairness of the comparison is ensured). The so-called flattening effect, discussed in \cite{Stojnicgscomp16}, appears as a consequence of the lifting procedure and tightens the corresponding comparisons from Section \ref{sec:gencon}.


\begin{figure}[htb]
\begin{minipage}[b]{.5\linewidth}
\centering
\centerline{\epsfig{figure=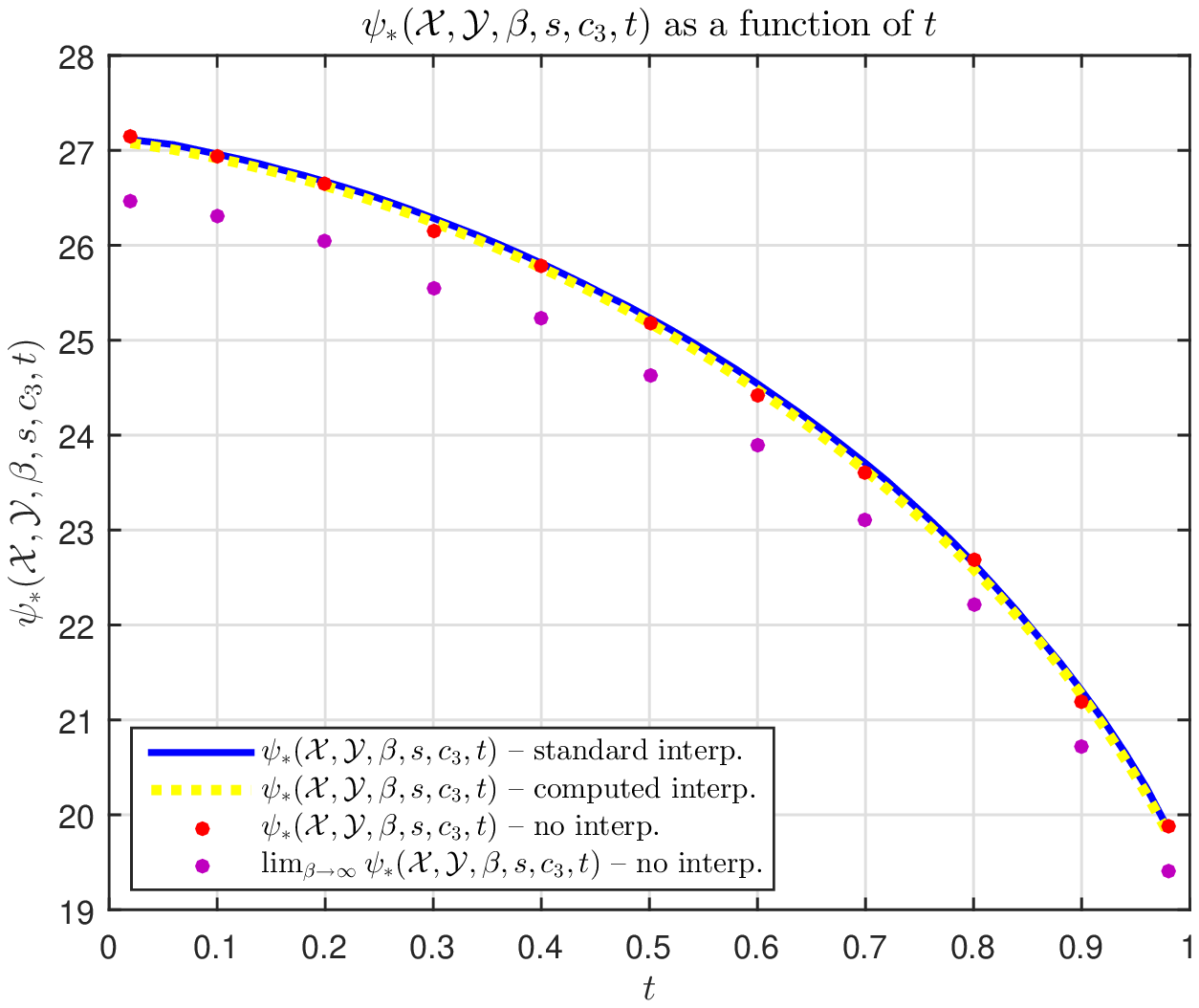,width=9cm,height=7cm}}
\end{minipage}
\begin{minipage}[b]{.5\linewidth}
\centering
\centerline{\epsfig{figure=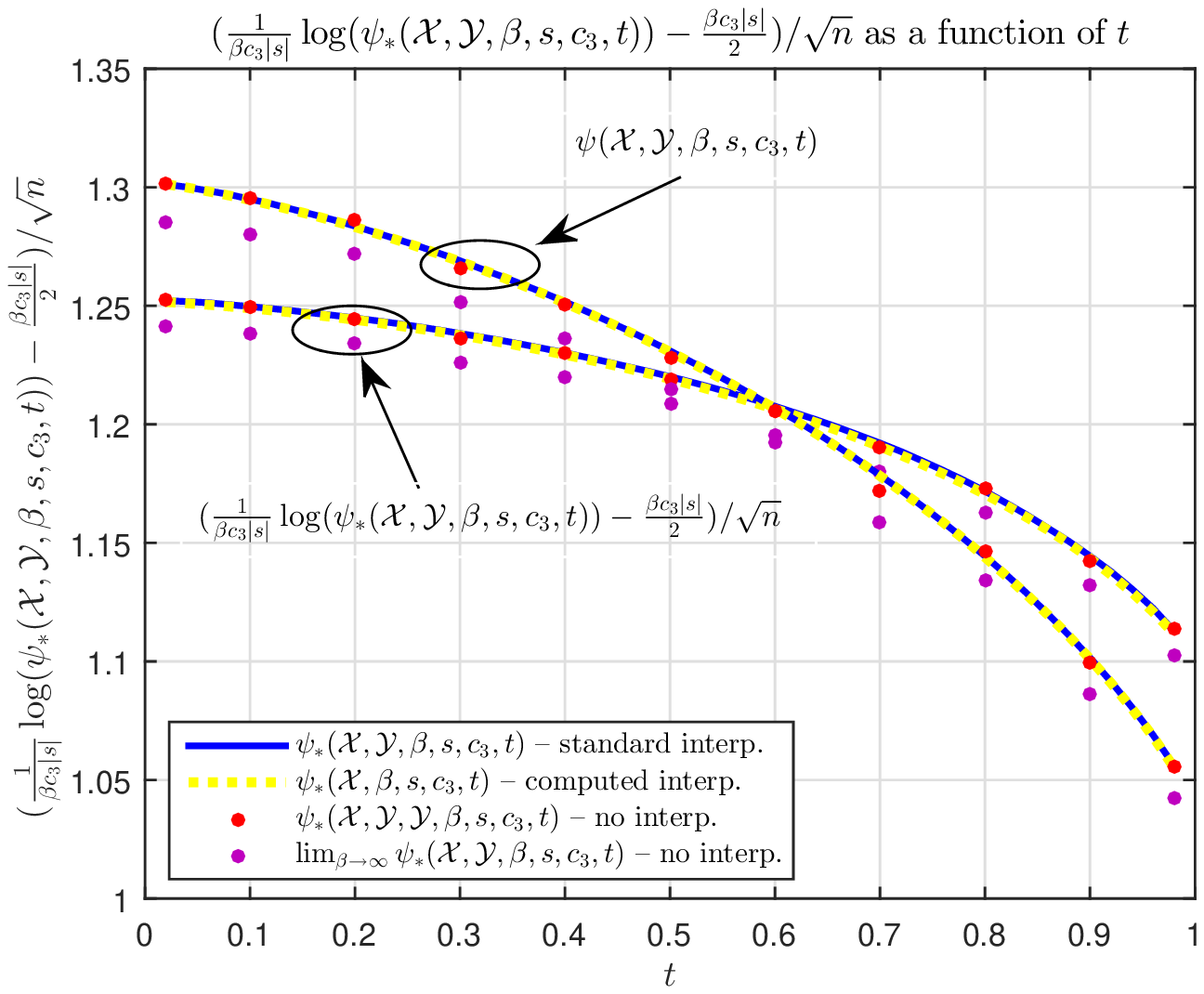,width=9cm,height=7cm}}
\end{minipage}
\caption{Left -- $\psi_*(\calX,\calY,\beta,s,c_3,t)$ as a function of $t$; $m=5$, $n=5$, $l=10$, $\calX=\calX^{+}$, $\calY=\calY^{+}$, $\beta=10$, $s=1$, $c_3=.1$; right -- comparison between adjusted $\psi_*(\calX,\calY,\beta,s,c_3,t)$ and $\psi_*(\calX,\calY,\beta,s,t)$ for $\beta=10$ (lifting versus no-lifting)}
\label{fig:liftedbetainfsplus1xnorm1psi}
\end{figure}

\begin{table}[h]
\caption{Simulated results --- $m=5$, $n=5$, $l=10$, $\calX=\calX^{+}$, $\calY=\calY^{+}$, $\beta=10$, $s=1$, $c_3=.1$}\vspace{.1in}
\hspace{-0in}\centering
{\small
\begin{tabular}{||c||c|c|c|c|c|c||}
\hline\hline
$ t$  &  $\frac{d\psi_*}{dt}$; (\ref{eq:liftgenanal11}) & $\frac{d\psi_*}{dt}$;  (\ref{eq:liftconalt2}) & $\psi_*$;  (\ref{eq:liftgenanal11}) and (\ref{eq:liftco1eq1}) & $\psi_*$; (\ref{eq:liftconalt2}) and (\ref{eq:liftco1eq1}) & $\psi_*$;  (\ref{eq:liftgenanal8})& $\lim_{\beta\rightarrow\infty}\psi_*$;  (\ref{eq:liftgenanal8})\\  \hline\hline
$ 0.1 $ & $ -2.4507 $ & $ -2.3279 $ & $\bl{\mathbf{ 26.9605 / 1.2497 }}$ & $\bl{\mathbf{ 26.9168 / 1.2490 }}$ & $\mathbf{ 26.9459 / 1.2494 }$& $\prp{\mathbf{ 26.2966 / 1.2385 }}$\\ \hline
$ 0.2 $ & $ -3.3235 $ & $ -3.2290 $ & $\bl{\mathbf{ 26.6710 / 1.2449 }}$ & $\bl{\mathbf{ 26.6268 / 1.2441 }}$ & $\mathbf{ 26.6599 / 1.2447 }$& $\prp{\mathbf{ 26.0429 / 1.2342 }}$\\ \hline
$ 0.3 $ & $ -4.3431 $ & $ -4.2313 $ & $\bl{\mathbf{ 26.2861 / 1.2384 }}$ & $\bl{\mathbf{ 26.2420 / 1.2376 }}$ & $\mathbf{ 26.1391 / 1.2358 }$& $\prp{\mathbf{ 25.5472 / 1.2256 }}$\\ \hline
$ 0.4 $ & $ -5.2185 $ & $ -5.1793 $ & $\bl{\mathbf{ 25.8102 / 1.2302 }}$ & $\bl{\mathbf{ 25.7632 / 1.2294 }}$ & $\mathbf{ 25.7952 / 1.2299 }$& $\prp{\mathbf{ 25.2189 / 1.2198 }}$\\ \hline
$ 0.5 $ & $ -6.2090 $ & $ -6.3270 $ & $\bl{\mathbf{ 25.2360 / 1.2201 }}$ & $\bl{\mathbf{ 25.1825 / 1.2192 }}$ & $\mathbf{ 25.1692 / 1.2189 }$& $\prp{\mathbf{ 24.6243 / 1.2091 }}$\\ \hline
$ 0.6 $ & $ -7.6488 $ & $ -7.5322 $ & $\bl{\mathbf{ 24.5415 / 1.2076 }}$ & $\bl{\mathbf{ 24.4755 / 1.2064 }}$ & $\mathbf{ 24.4265 / 1.2055 }$&  $\prp{\mathbf{ 23.9020 / 1.1958 }}$\\ \hline
$ 0.7 $ & $ -8.8894 $ & $ -9.1072 $ & $\bl{\mathbf{ 23.7031 / 1.1921 }}$ & $\bl{\mathbf{ 23.6265 / 1.1906 }}$ & $\mathbf{ 23.5953 / 1.1901 }$& $\prp{\mathbf{ 23.0935 / 1.1804 }}$\\ \hline
$ 0.8 $ & $ -11.6414 $ & $ -11.2737 $ & $\bl{\mathbf{ 22.6518 / 1.1718 }}$ & $\bl{\mathbf{ 22.5879 / 1.1705 }}$ & $\mathbf{ 22.6897 / 1.1726 }$& $\prp{\mathbf{ 22.2064 / 1.1629 }}$\\ \hline
$ 0.9 $ & $ -14.8057 $ & $ -14.8608 $ & $\bl{\mathbf{ 21.3083 / 1.1445 }}$ & $\bl{\mathbf{ 21.2622 / 1.1435 }}$ & $\mathbf{ 21.1912 / 1.1420 }$& $\prp{\mathbf{ 20.7216 / 1.1320 }}$
\\ \hline\hline
\end{tabular}
}
\label{tab:liftedbetainfsplus1xnorm1psi}
\end{table}

\subsubsection{$s=-1$ -- a lifted Gordon's fully bilinear minmax comparison}
\label{sec:liftbetainfsminus1}

Choosing $s=-1$ in (\ref{eq:liftliftbetainf1}) gives
\begin{eqnarray}\label{eq:liftliftbetainfsmin1}
 \log\lim_{\beta\rightarrow\infty} \xi_*(\calX,\calY,\beta,-1,\frac{c^{(s)}_3}{\beta})
 & = &  \log \mE_{G,u^{(4)}}\lp e^{c^{(s)}_3\max_{\x^{(i_1)}\in\calX} -\min_{\y^{(i_2)}\in\calY}\lp (\y^{(i_2)})^T
 G\x^{(i_1)}+\|\x^{(i_1)}\|_2\|\y^{(i_2)}\|_2 u^{(4)}\rp} \rp.\nonumber \\
\end{eqnarray}
Analogously to (\ref{eq:liftliftbetainfsplus2}) we now have
\begin{multline}\label{eq:liftliftbetainfsplus2}
\log \mE_{G,u^{(4)}}\lp e^{c^{(s)}_3\max_{\x^{(i_1)}\in\calX}- \max_{\y^{(i_2)}\in\calY}\lp (\y^{(i_2)})^T
 G\x^{(i_1)}+\|\x^{(i_1)}\|_2\|\y^{(i_2)}\|_2 u^{(4)}\rp} \rp  =
 \log\lim_{\beta\rightarrow\infty} \psi_*(\calX,\calY,\beta,-1,\frac{c^{(s)}_3}{\beta},1) \\
 \leq     \log\lim_{\beta\rightarrow\infty} \psi_*(\calX,\calY,\beta,-1,\frac{c^{(s)}_3}{\beta},t)
 \leq     \log\lim_{\beta\rightarrow\infty} \psi_*(\calX,\calY,\beta,-1,\frac{c^{(s)}_3}{\beta},0) \\
 =
\log \mE_{\u^{(2)},\h}\lp e^{c^{(s)}_3\lp \max_{\x^{(i_1)}\in\calX}-\max_{\y^{(i_2)}\in\calY}\lp \|\x^{(i_1)}\|_2(\y^{(i_2)})^T\u^{(2)}+\|\y^{(i_2)}\|_2\h^T\x^{(i_1)}\rp\rp} \rp.
\end{multline}
Taking beginning and end in (\ref{eq:liftliftbetainfsmin2}) establishes again exactly the same inequality as in the comparison principle we utilized in \cite{StojnicMoreSophHopBnds10} (as well as in e.g. \cite{StojnicLiftStrSec13,StojnicRicBnds13}). Namely, a Gordon's minmax principle was the key mechanism that we relied on in \cite{StojnicMoreSophHopBnds10} to obtain
\begin{multline}\label{eq:liftliftbetainfsmin2}
 \log \mE_{G,u^{(4)}} e^{c^{(s)}_3\max_{\x^{(i_1)}\in \calX}\min_{\y^{(i_2)}\in \calY} \lp (\y^{(i_2)})^T G\x^{(i_1)}+\|\x^{(i_1)}\|_2\|\y^{(i_2)}\|_2u^{(4)}\rp} \\
    \leq \log \mE_{\u^{(2)},\h} e^{c^{(s)}_3\max_{\x^{(i_1)}\in \calX}\min_{\y^{(i_2)}\in \calY} \lp \|\x^{(i_1)}\|_2(\y^{(i_2)})^T \u^{(2)}+ \|\y^{(i_2)}\|_2\h^T\x^{(i_1)}\rp}.
\end{multline}
Following the reasoning discussed above, one can think of (\ref{eq:liftliftbetainfsmin1}) and (\ref{eq:liftliftbetainfsmin2}) as being a lifted Gordon's minmax comparison principle (more on how useful this lifting strategy turns out to be can be found in, e.g. \cite{StojnicMoreSophHopBnds10,StojnicLiftStrSec13,StojnicRicBnds13}). As earlier, we emphasize that this form is only a special case of a much stronger concept introduced in Theorem \ref{thm:liftthm2}.

Following what we observed in \cite{Stojnicgscomp16}, when $\|\x^{(i_1)}\|_2=1,1\leq i_1\leq l$, and $\|\y^{(i_2)}\|_2=1,1\leq i_2\leq l$, we have the following rather elegant consequence of the above (basically for any $\beta$ and $s=1$)
\begin{multline}\label{eq:liftliftbetainfsmin3}
\frac{(c_3^{(s)})^2}{2}+c_3^{(s)}\mE_{G} \max_{\x^{(i_1)}\in \calX} s\max_{\y^{(i_2)}\in \calY}\lp (\y^{(i_2)})^T
 G\x^{(i_1)}\rp  =  \frac{(c_3^{(s)})^2}{2}+ \mE_{G} \log e^{c_3^{(s)} \max_{\x^{(i_1)}\in \calX} s\max_{\y^{(i_2)}\in \calY}\lp (\y^{(i_2)})^T
 G\x^{(i_1)}\rp}  \\
\leq  \log \mE_{G,u^{(4)}} e^{c_3^{(s)} \max_{\x^{(i_1)}\in \calX} s\max_{\y^{(i_2)}\in \calY}\lp (\y^{(i_2)})^T
 G\x^{(i_1)}+u^{(4)}\rp} \leq \log \mE_{G,u^{(4)}}\lp\sum_{i_1=1}^{l} \lp\sum_{i_2=1}^{l} e^{\beta \lp (\y^{(i_2)})^T
 G\x^{(i_1)}+u^{(4)}\rp}\rp^s \rp^{c_3} \\
=  \log \mE_{G,u^{(4)}} \psi_*(\calX,\calY,\beta,s,c_3,1) \leq  \log \mE_{G,u^{(4)},\u^{(2)},\h} \psi_*(\calX,\calY,\beta,s,c_3,t) \\
\leq  \log \mE_{G,u^{(4)},\u^{(2)},\h} \psi_*(\calX,\calY,\beta,s,c_3,0) = \log \mE_{G,u^{(4)},\u^{(2)},\h} \lp\sum_{i_1=1}^{l} \lp\sum_{i_2=1}^{l} e^{\beta \lp \|\x^{(i_1)}\|_2(\y^{(i_2)})^T
 \u^{(2)}+\|\y^{(i_2)}\|_2\h^T\x^{(i_1)}\rp}\rp^s \rp^{c_3}.\\
\end{multline}
From (\ref{eq:liftliftbetainfsmin3}) we find
\begin{eqnarray}\label{eq:liftliftbetainfsmin3a}
\mE_{G} \max_{\x^{(i_1)}\in \calX} s\max_{\y^{(i_2)}\in \calY}\lp (\y^{(i_2)})^T
 G\x^{(i)}\rp  & \leq & \frac{1}{c_3^{(s)}} \log \mE_{G,u^{(4)},\u^{(2)},\h} \psi_*(\calX,\calY,\beta,s,c_3,t)-\frac{c_3^{(s)}}{2} \nonumber \\
& = & \frac{1}{\beta c_3} \log \mE_{G,u^{(4)},\u^{(2)},\h} \psi_*(\calX,\calY,\beta,s,c_3,t)-\frac{\beta c_3}{2}.
\end{eqnarray}
For $t=1$, (\ref{eq:liftliftbetainfsmin3}) and (\ref{eq:liftliftbetainfsmin3a}) give
\begin{multline}\label{eq:liftliftbetainfsmin4}
\mE_{G} \max_{\x^{(i_1)}\in \calX} s\max_{\y^{(i_2)}\in \calY}\lp (\y^{(i_2)})^T
 G\x^{(i)}\rp   \leq  \frac{1}{c_3^{(s)}} \log \mE_{G,u^{(4)},\u^{(2)},\h} \psi_*(\calX,\calY,\beta,s,c_3,0)-\frac{c_3^{(s)}}{2}  \\
=  \frac{1}{c_3^{(s)}} \log \mE_{G,u^{(4)},\u^{(2)},\h} \lp\sum_{i_1=1}^{l} \lp\sum_{i_2=1}^{l} e^{\beta \lp \|\x^{(i_1)}\|_2(\y^{(i_2)})^T
 \u^{(2)}+\|\y^{(i_2)}\|_2\h^T\x^{(i_1)}\rp}\rp^s \rp^{c_3}-\frac{c_3^{(s)}}{2} \\
=  \frac{1}{\beta c_3} \log \mE_{G,u^{(4)},\u^{(2)},\h} \lp\sum_{i_1=1}^{l} \lp\sum_{i_2=1}^{l} e^{\beta \lp \|\x^{(i_1)}\|_2(\y^{(i_2)})^T
 \u^{(2)}+\|\y^{(i_2)}\|_2\h^T\x^{(i_1)}\rp}\rp^s \rp^{c_3}-\frac{\beta c_3}{2}.
\end{multline}
Finally, for $\beta\rightarrow\infty$ (and $c_3=\frac{c_3^{(2)}}{\beta}$) we have
\begin{multline}\label{eq:liftliftbetainfsmin5}
\mE_{G} \max_{\x^{(i_1)}\in \calX} s\max_{\y^{(i_2)}\in \calY}\lp (\y^{(i_2)})^T
 G\x^{(i)}\rp \\ \leq  \frac{1}{c_3^{(s)}} \log \mE_{u^{(2)},\h} \lp e^{c_3^{(s)}\max_{\x^{(i_1)}\in\calX} s \max_{\y^{(i_2)}\in\calY}\lp \|\x^{(i_1)}\|_2(\y^{(i_2)})^T
 \u^{(2)}+\|\y^{(i_2)}\|_2\h^T\x^{(i_1)}\rp}\rp-\frac{c_3^{(s)}}{2}.
\end{multline}
 Connecting the first inequality in (\ref{eq:liftliftbetainfsmin3}) and (\ref{eq:liftliftbetainfsmin2}) ensures that (\ref{eq:liftliftbetainfsmin5}) in particular remains true even when $s=-1$. Of course, (\ref{eq:liftliftbetainfsmin5}) is
basically one of the key features of the mechanisms we introduced and utilized in e.g. \cite{StojnicMoreSophHopBnds10,StojnicLiftStrSec13,StojnicRicBnds13}. One can also take the $\beta\rightarrow\infty$ limit for any $t$ to obtain for any sign $s$ a bit stronger (though probably often less useful)
\begin{multline}\label{eq:liftliftbetainfsmin6}
\mE_{G} \max_{\x^{(i_1)}\in \calX} s\max_{\y^{(i_2)}\in \calY}\lp (\y^{(i_2)})^T
 G\x^{(i)}\rp \\ \leq  \frac{1}{c_3^{(s)}} \log \mE_{u^{(2)},\h} \lp e^{c_3^{(s)}\max_{\x^{(i_1)}\in\calX} s \max_{\y^{(i_2)}\in\calY}f_{1,h}(\x^{(i_1)},\y^{(i_2)},G,u^{(4)},\u^{(2)},\h,t)}\rp-\frac{c_3^{(s)}}{2},
\end{multline}
where
\begin{multline}\label{eq:liftliftbetainfsmin7}
f_{1,h}(\x^{(i_1)},\y^{(i_2)},G,u^{(4)},\u^{(2)},\h,t) = \sqrt{t}(\y^{(i_2)})^TG\x^{(i_1)}+ \sqrt{1-t}\|\x^{(i_1)}\|_2(\y^{(i_2)})^T
 \u^{(2)} \\+\sqrt{t}\|\x^{(i_1)}\|_2\|\y^{(i_2)}\|_2u^{(4)}+\sqrt{1-t}\|\y^{(i_2)}\|_2\h^T\x^{(i_1)}.
\end{multline}

\textbf{\underline{\emph{Numerical results}}}

Figure \ref{fig:liftedbetainfsmin1xnorm1psi} and Table \ref{tab:liftedbetainfsmin1xnorm1psi} contain the results that we obtained through the numerical simulations. All parameters are again the same as earlier, including $\beta=10$ which again in a way emulates $\beta\rightarrow\infty$ and $c_3=.1$ (everything is averaged over a set of $5e4$ experiments). We again observe from both, Figure \ref{fig:liftedbetainfsmin1xnorm1psi} and Table \ref{tab:liftedbetainfsmin1xnorm1psi}, that there is a solid agreement between all the presented results (with $\beta=10$ again being a pretty good approximation of $\beta\rightarrow\infty$). As earlier, the right part of the figure shows again appearance of the flattening effect which is one of the key consequences of the lifting procedure. Clearly, this then tightens the corresponding comparisons from Section \ref{sec:gencon}. We should also add that $c_3=.1$ is not necessarily the best value that one can take to have the flattening effect at its full power (both, here as well as when we discussed the lifting of the Slepian's max principle earlier). However, we selected a value that is reasonably close to the one that would tighten the corresponding comparisons from Section \ref{sec:gencon} the most.


\begin{figure}[htb]
\begin{minipage}[b]{.5\linewidth}
\centering
\centerline{\epsfig{figure=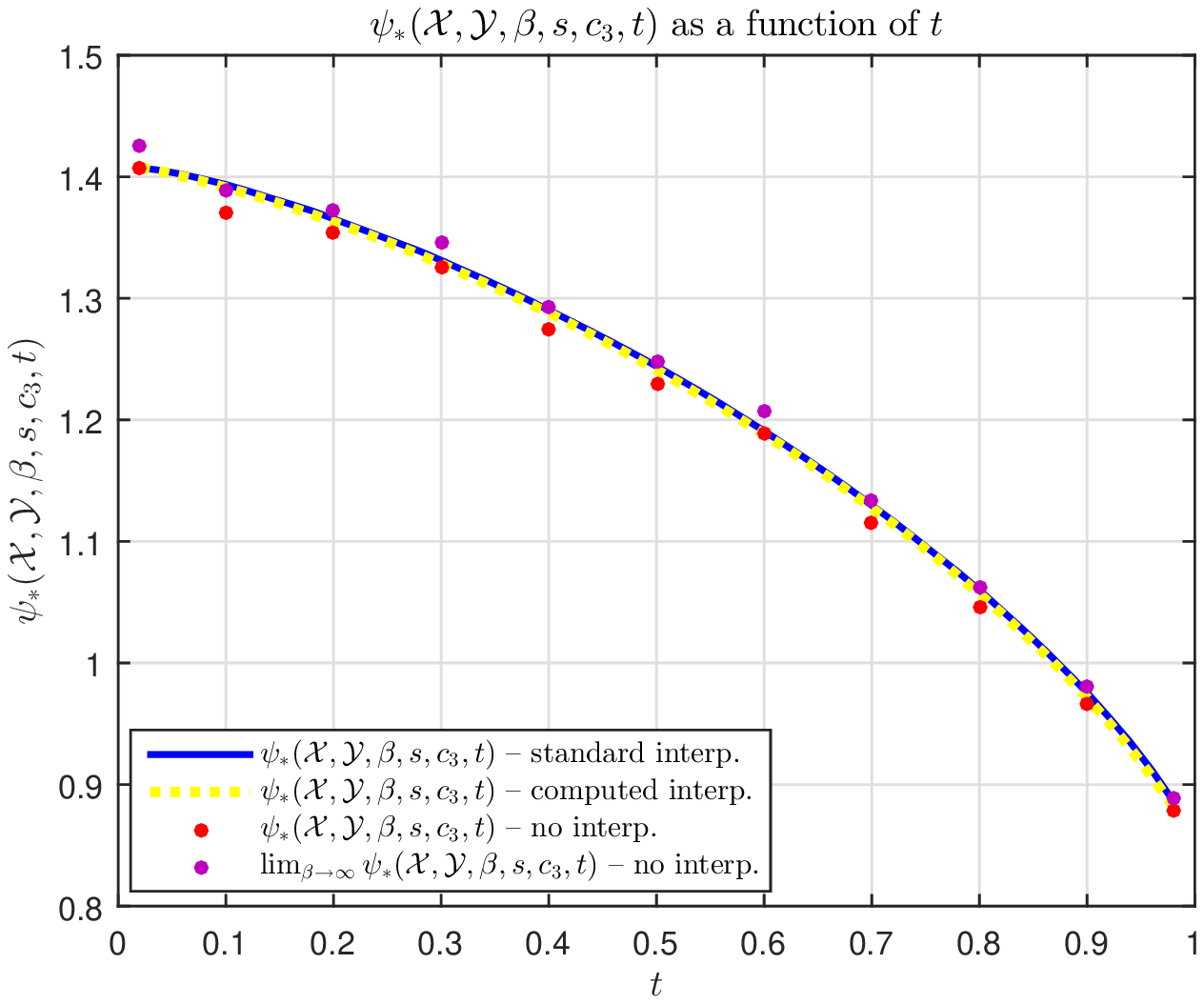,width=9cm,height=7cm}}
\end{minipage}
\begin{minipage}[b]{.5\linewidth}
\centering
\centerline{\epsfig{figure=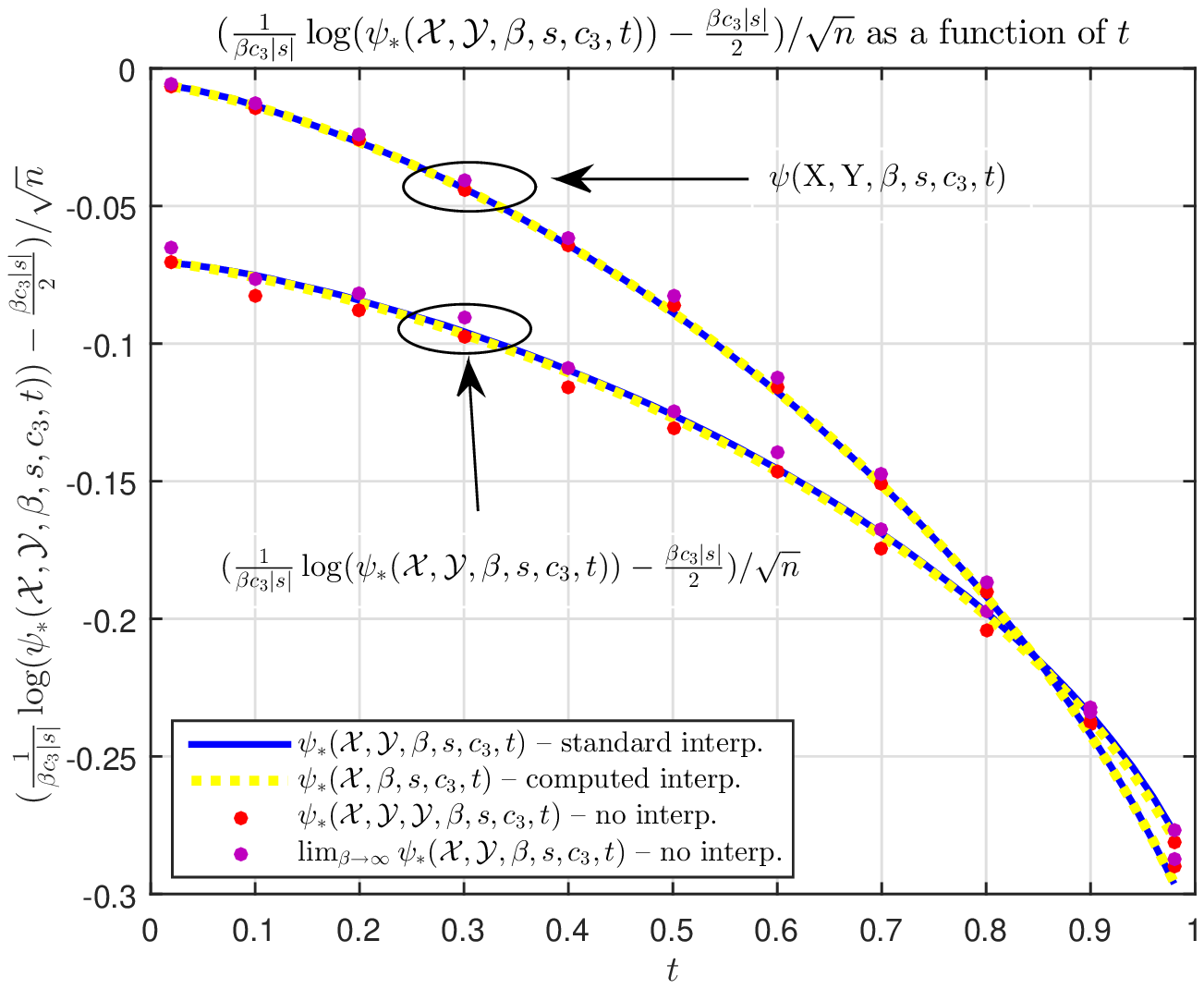,width=9cm,height=7cm}}
\end{minipage}
\caption{Left -- $\psi_*(\calX,\calY,\beta,s,c_3,t)$ as a function of $t$; $m=5$, $n=5$, $l=10$, $\calX=\calX^{+}$, $\calY=\calY^{+}$, $\beta=10$, $s=-1$, $c_3=.1$; right -- comparison between adjusted $\psi_*(\calX,\calY,\beta,s,c_3,t)$ and $\psi_*(\calX,\calY,\beta,s,t)$ for $\beta=10$ (lifting versus no-lifting)}
\label{fig:liftedbetainfsmin1xnorm1psi}
\end{figure}

{\footnotesize
\begin{table}[h]
\caption{Simulated results --- $m=5$, $n=5$, $l=10$, $\calX=\calX^{+}$, $\calY=\calY^{+}$, $\beta=10$, $s=-1$, $c_3=.1$}\vspace{.1in}
\hspace{-0in}\centering
{\small
\begin{tabular}{||c||c|c|c|c|c|c||}\hline\hline
$ t$  &  $\frac{d\psi_*}{dt}$; (\ref{eq:liftgenanal11}) & $\frac{d\psi_*}{dt}$;  (\ref{eq:liftconalt2}) & $\psi_*$;  (\ref{eq:liftgenanal11}) and (\ref{eq:liftco1eq1}) & $\psi_*$; (\ref{eq:liftconalt2}) and (\ref{eq:liftco1eq1}) & $\psi_*$;  (\ref{eq:liftgenanal8})& $\lim_{\beta\rightarrow\infty}\psi_*$;  (\ref{eq:liftgenanal8})\\  \hline\hline
$ 0.1 $ & $ -0.2114 $ & $ -0.2353 $ & $\bl{\mathbf{ 1.3936 / -0.0752 }}$ & $\bl{\mathbf{ 1.3912 / -0.0759 }}$ & $\mathbf{ 1.3708 / -0.0826 }$& $\prp{\mathbf{ 1.3889 / -0.0767 }}$\\ \hline
$ 0.2 $ & $ -0.3276 $ & $ -0.3069 $ & $\bl{\mathbf{ 1.3654 / -0.0843 }}$ & $\bl{\mathbf{ 1.3630 / -0.0851 }}$ & $\mathbf{ 1.3537 / -0.0882 }$& $\prp{\mathbf{ 1.3729 / -0.0819 }}$\\ \hline
$ 0.3 $ & $ -0.3804 $ & $ -0.3735 $ & $\bl{\mathbf{ 1.3313 / -0.0956 }}$ & $\bl{\mathbf{ 1.3284 / -0.0966 }}$ & $\mathbf{ 1.3259 / -0.0974 }$& $\prp{\mathbf{ 1.3454 / -0.0909 }}$\\ \hline
$ 0.4 $ & $ -0.4439 $ & $ -0.4279 $ & $\bl{\mathbf{ 1.2903 / -0.1096 }}$ & $\bl{\mathbf{ 1.2879 / -0.1105 }}$ & $\mathbf{ 1.2736 / -0.1155 }$& $\prp{\mathbf{ 1.2933 / -0.1086 }}$\\ \hline
$ 0.5 $ & $ -0.4930 $ & $ -0.4961 $ & $\bl{\mathbf{ 1.2441 / -0.1259 }}$ & $\bl{\mathbf{ 1.2412 / -0.1270 }}$ & $\mathbf{ 1.2301 / -0.1310 }$& $\prp{\mathbf{ 1.2489 / -0.1242 }}$\\ \hline
$ 0.6 $ & $ -0.5535 $ & $ -0.5604 $ & $\bl{\mathbf{ 1.1908 / -0.1455 }}$ & $\bl{\mathbf{ 1.1882 / -0.1465 }}$ & $\mathbf{ 1.1887 / -0.1463 }$& $\prp{\mathbf{ 1.2079 / -0.1391 }}$\\ \hline
$ 0.7 $ & $ -0.6523 $ & $ -0.6363 $ & $\bl{\mathbf{ 1.1297 / -0.1691 }}$ & $\bl{\mathbf{ 1.1275 / -0.1699 }}$ & $\mathbf{ 1.1156 / -0.1747 }$& $\prp{\mathbf{ 1.1334 / -0.1676 }}$\\ \hline
$ 0.8 $ & $ -0.7503 $ & $ -0.7578 $ & $\bl{\mathbf{ 1.0599 / -0.1976 }}$ & $\bl{\mathbf{ 1.0567 / -0.1989 }}$ & $\mathbf{ 1.0452 / -0.2038 }$& $\prp{\mathbf{ 1.0615 / -0.1969 }}$\\ \hline
$ 0.9 $ & $ -0.9074 $ & $ -0.9375 $ & $\bl{\mathbf{ 0.9755 / -0.2347 }}$ & $\bl{\mathbf{ 0.9715 / -0.2365 }}$ & $\mathbf{ 0.9672 / -0.2385 }$& $\prp{\mathbf{ 0.9806 / -0.2324 }}$
\\ \hline\hline
\end{tabular}
}
\label{tab:liftedbetainfsmin1xnorm1psi}
\end{table}
}

\section{Conclusion}
\label{sec:liftconc}

A collection of very powerful statistical comparison results is presented. We first introduced a general comparison concept that we call fully bilinear. Then we showed how such a concept can be upgraded through a lifting procedure. All our theoretical findings we then complemented with an extensive set of numerical results. These were obtained trough simulations and are observed to be in an excellent agreement with the theoretical predictions. Moreover, for both, the general and the lifted strategy, we showed that they contain as special cases the well known Slepian's max and Gordon's minmax comparison principles. Since many of the results that we created in various fields of mathematics in recent years utilize as starting points these well-known principles, the results presented here make all of them substantially more general and fully self-contained.

The mechanisms that we presented here seem like a very powerful self-sustainable tool which can be used for various extensions. Typically these extensions require a few rather routine modifications of the main concepts presented here and in a couple of our earlier works. For the extensions that we find to be of particular interest we will present the needed modifications as well as the final results that one can obtain through them in a few separate papers.

\begin{singlespace}
\bibliographystyle{plain}
\bibliography{gscompyxRefs}
\end{singlespace}

\end{document}